\documentclass{amsart}

\usepackage[colorlinks=true, urlcolor=black, citecolor=black, linkcolor=black, hyperfootnotes=true]{hyperref}
\usepackage{amssymb}
\usepackage{aliascnt}%for theorem counters
\usepackage{mathscinet}
\usepackage{tikz}%for Hasse diagram
\usepackage{graphicx}%for Hasse diagram
\usepackage{accents}
\newcommand*{\dt}[1]{%
  \accentset{\mbox{\large\bfseries .}}{#1}}

\numberwithin{equation}{section}

\newtheorem{thm}{Theorem}[section]

\newaliascnt{prp}{thm}
\newtheorem{prp}[prp]{Proposition}
\aliascntresetthe{prp}

\newaliascnt{cor}{thm}
\newtheorem{cor}[cor]{Corollary}
\aliascntresetthe{cor}

\theoremstyle{definition}

\newaliascnt{dfn}{thm}
\newtheorem{dfn}[dfn]{Definition}
\aliascntresetthe{dfn}

\newaliascnt{xpl}{thm}
\newtheorem{xpl}[xpl]{Example}
\aliascntresetthe{xpl}

\newaliascnt{rmk}{thm}
\newtheorem{rmk}[rmk]{Remark}
\aliascntresetthe{rmk}

\author{Tristan Bice}
\thanks{The author is supported by Cardinal Stefan Wyszy\'nski University in Warsaw, Poland.}
\keywords{Wallman duality, uniformity, nearness, poset, Cauchy filter, compact, locally compact, regular space, completely metrisable}
\subjclass[2010]{06A06, 06A07, 54D10, 54D30, 54D35, 54D45, 54D70, 54D80, 54E15, 54E17, 54E35, 54E45, 54E50, 54H10}
\title{Nearness Posets}

\begin{document}

\begin{abstract}
We extend nearness frames to posets representing bases and even subbases of $T_1$ spaces.  This allows us to put a classic duality due to Wallman, between compact $T_1$ spaces and abstract simplicial complexes, into a general nearness framework.  Within this framework we extend Wallman's duality to locally compact $T_1$ spaces and, with further modifications, to completely metrisable spaces.  Moreover, we provide an elementary sublocale-free version of an admissibility condition due to Picado and Pultr and show how it strengthens Wallman's admissibility condition.
\end{abstract}

\maketitle

\section*{Introduction}

The story really begins 80 years ago with a paper of Wallman, namely \cite{Wallman1938}.  Most attention was focused on the earlier lattice theoretic part of \cite{Wallman1938} which, together with Stone's famous papers around the same time, gave birth to the field of point-free topology.  However, in the latter part of \cite{Wallman1938}, Wallman showed that more general posets arising from abstract simplicial complexes also provide a point-free description of compact $T_1$ spaces.  Here the vertices of the complex correspond to a subbasis of the space, while the simplices in the complex determine the covers of the space.  Thus the complex could be considered as a point-free `generalised uniform space'.  This duality was well ahead of its time and it would be many years before point-free uniform or nearness structures were considered again.

On the pointed side of things, however, entourage uniformities were introduced by Weil around the same time, while Tukey introduced covering uniformities in \cite{Tukey1940}.  These were significantly generalised by Morita in a series of papers \cite{Morita1951} which were somewhat overlooked at the time, despite their importance.  Indeed, to quote Collins in \cite{CollinsMorita1998}, Morita's papers formed an `extended rewriting of the basis of general topology' via generalised uniformities.  Our goal in the present paper is to explore similar ideas but in a point-free context.

The word `basis' here really is the key.  Indeed, Morita made two key observations about covering uniformities:
\begin{enumerate}
\item It is often best to weaken the star-refinement axiom.
\item It usually suffices to work with a basis rather than all open sets.
\end{enumerate}
The first observation was made independently several times thereafter, although in equivalent but superficially different contexts, like Herrlich's nearness spaces and Kat\v{e}tov's merotopic spaces (see \cite{BentleyHerrlichHusek1998} for more on the history and interrelatedness of these structures).  Once this equivalence became apparent, `nearness spaces' gradually came to be synonymous with Morita's `generalised uniform spaces'.

More recently, uniform and nearness structures have been considered in the point free context, namely on frames (see \cite{PicadoPultr2012}).  However, what appears to have escaped attention is the fact that Morita's second observation is still valid in the point free context.  More precisely, instead of considering covers on a frame, representing the entire collection of open sets of some space, we can work with a poset, representing a mere basis of some space, leaving the covers to carry most of the topological information, just like Wallman did in \cite{Wallman1938}.

The significance of this observation is easy to overlook, but its practical and theoretical implications should not be underestimated.  On the practical side of things, by working posets/bases we are dealing with much smaller, often countable objects.  These will generally be easier to construct, often in more finite combinatorial ways, like from inverse systems of graphs in a manner similar to \cite{DebskiTymchatyn2018}.  On the theoretical side of things, this illustrates that covers really do contain a lot of information about the space.

Indeed, on their own, posets arising from arbitrary bases of topological spaces do not retain much information about the space at all, excepting e.g. isolated points.  For example, all second countable regular spaces without isolated points have bases which are isomorphic as posets (to the unique countable atomless Boolean algebra).  So there is no way to distinguish, say, the unit interval from the Cantor space given just the order structure of some arbitrary basis.

To rectify this, we considered slightly stronger structures in locally compact Hausdorff spaces, namely bases closed under finite unions (see \cite{BiceStarling2018}) or bases together with the compact containment relation (see \cite{BiceStarling2018b}).  However, by obtaining the requisite extra structure from covers instead we are able to recover a much larger class of spaces than the locally compact Hausdorff ones.  In fact, we can recover all $T_1$ spaces, including the non-sober ones.  Moreover, as a bonus we get some additional uniform structure from which to define things like uniform continuity (although this will not be the focus of the present paper).

\section*{Outline}

First, we set out some basic notation and terminology in \autoref{Preliminaries}.  In \autoref{CauchyvsRound}, we then examine some elementary properties of Cauchy and round up-sets.  The more standard approach would be to consider regular filters, which would be fine for bases in $T_3$ spaces.  However, we will be working with subbases in $T_1$ spaces, which is why we must consider more general round up-sets.  Under suitable extra conditions, they turn out to be the same thing, as becomes clear later on.  Following that, we introduce the spectrum of round Cauchy up-sets in \autoref{TheSpectrum}.  In particular, in \autoref{T1}, we note that the spectrum allows us to recover any $T_1$ space from an arbitrary subbasis.

We come to our first duality in \autoref{WallmanDuality}.  The results in this sections are essentially all from \cite{Wallman1938}, but reformulated in a way that makes it clearer how they relate to nearness structures and also easier to generalise.  The main result is summarised in \autoref{WallmanDualitySummary}, which we repeat here:
\[\textbf{Wallman posets are dual to subbases of compact $T_1$ spaces}.\]
Here, Wallman posets are structures of the form $(\mathbb{P},\leq,\Theta)$ where $\Theta$ consists of finite subsets of $\mathbb{P}$ and $\Theta$ determines the ordering $\leq$ in a natural way (see \autoref{WallmanAdmissibleEquivalent}).  For details, see the comments following \autoref{WallmanDualitySummary}.

Before moving on, we digress to examine near subsets in \autoref{NearSubsets}.  This is necessary in order to define appropriate `restrictions' $\Theta|p$ to elements of $\mathbb{P}$.  Using these restrictions, we extend Wallman's duality to locally compact $T_1$ spaces.  The main result is summarised in \autoref{GeneralisedWallmanDualitySummary}, we which again repeat here:
\begin{center}
\textbf{\upshape Generalised Wallman posets are dual to\\ relatively compact subbases of locally compact $T_1$ spaces}.
\end{center}
Basically, $(\mathbb{P},\leq,\Theta)$ is a generalised Wallman poset if the restrictions $\Theta|p$ are finite and can be patched together to obtain $\Theta$, which is again required to determine the ordering $\leq$.  For the technical details, see the comments following \autoref{GeneralisedWallmanDualitySummary}.

From \autoref{BasicWallmanDuality} onwards, we move from general subbases to focus on bases.  This is done in order to handle more general families of covers, like in the main result summarised in \autoref{BasicWallmanDualitySummary}, specifically:
\begin{center}
\textbf{\upshape Abstract nearly finite proper Wallman admissible filters $\Theta\subseteq\mathcal{P}(\mathbb{P})$ are dual to concrete $*$-coinitial filters $\Theta\subseteq\mathcal{C}_X(\mathbb{P})$ where $\mathbb{P}\subseteq\mathcal{O}(X)$ is a basis of relatively compact sets on a locally compact $T_1$ space $X\neq\emptyset$}.
\end{center}
The key point is that $*$-coinitial families of covers are general enough to include examples like the uniform covers of a locally compact metric space.

We temporarily digress again in \autoref{OrderCovers} to examine a notion of cover that applies to general posets.  This is done in order to extend an admissibility condition due to Picado and Pultr from frames to posets, which we examine in \autoref{Picado-Pultr Admissibility}.  Even in the case of frames, our condition is somewhat simpler than the original in that it avoids any mention of sublocales and can be stated concisely as
\[\tag{Picado-Pultr Admissible}p\leq q\qquad\Leftrightarrow\qquad\forall C\in\Theta\ (\bigvee_{p\neq c\in C}c\vee q=1)\]
(note that $\Rightarrow$ is immediate if $\Theta$ consists of covers, it is the $\Leftarrow$ part that is the key).  In \autoref{PPEquiv}, we show that this is indeed equivalent to the sublocale version.

In \autoref{UniformlyBelow}, we examine the `uniformly below' relation $\vartriangleleft$, bringing us closer to the theory in \cite{PicadoPultr2012} (although we are still working with more general posets rather than frames).  Finally, we examine $\vartriangleleft$-regularity and the slightly weaker notion of star-regularity.  In \autoref{AdmissibleDirectedRegularisation}, we then show how any admissible $\Theta$ can be turned into a regular admissible family $\Theta^\mathsf{R}$.  This allows us to prove our final duality for completely metrisable spaces, summarised in \autoref{CompletelyMetrisableDualitySummary} and repeated here:
\begin{center}
\textbf{\upshape Abstract countable admissible filters $\Theta\subseteq\mathcal{P}(\mathbb{P})$ are dual to concrete countable complete locally uniform compatible filters $\Theta\subseteq\mathcal{C}_X(\mathbb{P})$ where $\mathbb{P}\subseteq\mathcal{O}(X)$ is a basis of a completely metrisable space $X$}.
\end{center}

\section{Preliminaries}\label{Preliminaries}

For a relation $\prec$ on a set $\mathbb{P}$ and $S\subseteq\mathbb{P}$, define
\[S^\prec=\{p\in\mathbb{P}:\exists s\in S\ (s\prec p)\}.\]
Further define the following conditions on $S$.
\begin{align}
\tag{$\prec$-Closed}S^\prec\subseteq S.\\
\tag{$\prec$-Regular}S^\prec\supseteq S.\\
\tag{$\prec$-Coinitial}\mathbb{P}\subseteq S^\prec.\\
\tag{$\prec$-Cofinal}\mathbb{P}\subseteq S^\succ.
\end{align}
We call any $\prec$-closed $S\subseteq\mathbb{P}$ an \emph{up-set} when the relation is clear.

As in \cite[VIII.2.1]{PicadoPultr2012}, we extend $\prec$ to $\mathcal{P}(\mathbb{P})=\{S:S\subseteq\mathbb{P}\}$ by defining
\[\tag{$\prec$-Refinement}R\prec S\quad\Leftrightarrow\quad R\subseteq S^\succ\quad\Leftrightarrow\quad\forall r\in R\ \exists s\in S\ (r\prec s).\]
When $R\prec S$, we say that $R$ \emph{refines} $S$.  Note that above we are using the standard convention of writing $\succ$ for the opposite relation of $\prec$, i.e. $s\succ r$ means $r\prec s$.  For subsets, however, $S\succ R$ is not the same as $R\prec S$, indeed
\begin{align*}
R\prec S\qquad&\Leftrightarrow\qquad R\subseteq S^\succ\\
\text{while}\qquad S\succ R\qquad&\Leftrightarrow\qquad S\subseteq R^\prec.
\end{align*}
Also note that we abuse notation by identifying elements with singleton subsets whenever there is no risk of confusion, e.g.
\[R\prec q\quad\text{really means}\quad R\prec\{q\}.\]

Let us denote the finite subsets of $\mathbb{P}$ (including the empty set $\emptyset$) by
\[\mathcal{F}(\mathbb{P})=\{F\subseteq\mathbb{P}:F\text{ is finite}\}.\]
Consider the following conditions on $S\subseteq\mathbb{P}$.
\begin{align}
\tag{$\prec$-Filter}F\in\mathcal{F}(S)\qquad&\Leftrightarrow\qquad\exists s\in S\ (F\succ s).\\
\tag{$\prec$-Directed}F\in\mathcal{F}(S)\qquad&\Rightarrow\qquad\exists s\in S\ (F\succ s).
\end{align}
When the relation is clear, we just refer to \emph{filters} and \emph{directed} subsets.  Note filters are just directed up-sets and directed subsets are always non-empty, as $\emptyset\in\mathcal{F}(S)$.

When $\prec$ is transitive and reflexive, we call $\prec$ a \emph{preorder} and $(\mathbb{P},\prec)$ a \emph{proset}.  The extension of $\prec$ to $\mathcal{P}(\mathbb{P})$ is then also a preorder.  When $\prec$ is also antisymmetric, we call $\prec$ a \emph{partial order} and $(\mathbb{P},\prec)$ a \emph{poset}.  However, the extension of $\prec$ to $\mathcal{P}(\mathbb{P})$ is generally not antisymmetric (in fact this only happens when $\prec$ is $=$).

A \emph{topology} on a set $X$ is a collection of open subsets $\mathcal{O}(X)\subseteq\mathbb{P}(X)$ that is closed under taking arbitrary unions and finite intersections.  In particular, $\emptyset=\bigcup\emptyset$ and $X=\bigcap\emptyset$ are always open.  We call $\mathbb{P}\subseteq\mathcal{O}(X)$ a \emph{basis} if every $p\in\mathcal{O}(X)$ is a union of elements of $\mathbb{P}$.  We call $\mathbb{P}\subseteq\mathcal{O}(X)$ a \emph{subbasis} if the finite intersections of elements of $\mathbb{P}$ form a basis.  Again note that we allow empty unions and intersections, so $\mathbb{P}=\emptyset$ is both a basis for the empty space $X=\emptyset$ and a subbasis for the one-point space $X=\{x\}$ (where \emph{space} here means \emph{topological space}, i.e. a set $X$ together with a topology $\mathcal{O}(X)$, which is unique for the empty and one-point spaces).

\section{Cauchy vs Round}\label{CauchyvsRound}

Our primary goal is to investigate posets together with some additional structure coming from a distinguished family of subsets.  Accordingly, throughout we assume
\[\textbf{$(\mathbb{P},\leq,\Theta)$ is structure where $(\mathbb{P},\leq)$ is a poset and $\Theta\subseteq\mathcal{P}(\mathbb{P})$}.\]
So $\Theta$ is a very general kind of `nearness'.  We leave `nearness' itself as a vague term for $\Theta$ satisfying some `nice' properties (although `near subset' will take on a precise meaning in \autoref{NearSubsets}).  Exactly what these properties might be will depend on the kinds of spaces and coverings one is dealing with.  The goal of the present paper is to explore these various properties, how they relate to each other and which ones are the weakest necessary to prove the results we want.

Our $\mathbb{P}$ is meant to represent a basis of some space $X$, $\leq$ is meant to represent inclusion $\subseteq$ and $\Theta$ is meant to represent a family of basic open covers of $X$.  We wish to recover the points of this hypothetical $X$ via their basic open neighbourhood filters.  Arbitrary filters in $\mathbb{P}$ will not do, e.g. these could correspond to the neighbourhoods of a subset rather than a single point.  However, $\Theta$ allows us to identify precisely which filters should really correspond to points.

Firstly, as every $C\in\Theta$ represents a cover of the space, every neighbourhood filter of $x\in X$ should contain some element of $C$.  These are the `Cauchy' filters, but again arbitrary Cauchy filters will not do, e.g. in a metric space different Cauchy filters can converge to the same point $x$.  The neighbourhood filter of $x$ is the smallest of these, which suggests we should consider minimal Cauchy filters.

In fact, at first we will even consider more general subbases.  Then the above remarks still apply, the only difference is that the subbasic neighbourhoods of a point may only form an up-set, not a filter.  Accordingly, we start by examining minimal or `round' Cauchy up-sets.

\begin{dfn}
Consider the following conditions on $S\subseteq\mathbb{P}$.
\begin{align}
\label{ThetaCauchy}\tag{$\Theta$-Cauchy}&\forall C\in\Theta\ (C\cap S\neq\emptyset).\\
\label{ThetaRound}\tag{$\Theta$-Round}\forall s\in S\ &\exists C\in\Theta\ (C\cap S\leq s).
\end{align}
\end{dfn}

\begin{rmk}
Here we are following the terminology of \cite[Definition 2.4 (b)]{CollinsHendrie2000}, although the definitions there are restricted to filters.  Also, the round subsets there are further required to satisfy $C\cap S\neq\emptyset$, presumably so that non-Cauchy subsets are not automatically round.  We prefer the above definition so that roundness is preserved by subsets, just as Cauchyness is preserved by supersets.  Cauchy filters are defined likewise in \cite[Definition 2.1]{Morita1989} but round filters are instead called `weak star-filters' in \cite[Definition 2.10]{Morita1989}.  This differs from Morita's earlier terminology in \cite{Morita1951} where a stronger notion of Cauchy filter is given and round filters do not appear at all \textendash\, round filters are originally from \cite{Harris1971}, although even there the focus was on a slightly stronger notion for Hausdorff spaces.
\end{rmk}

If $\Theta$ is an up-set, every $\Theta$-round $\Theta$-Cauchy subset is also an up-set.

\begin{prp}\label{Upset=>Upset}
If $R\subseteq\mathbb{P}$ is $\Theta$-round and $\Theta$-Cauchy,
\[\Theta\text{ is an up-set}\qquad\Rightarrow\qquad R\text{ is an up-set}.\]
\end{prp}

\begin{proof}
Say $s\geq r\in R$.  As $R$ is $\Theta$-round, we have some $C\in\Theta$ with $C\cap R\leq r\leq s$ and hence $C\leq D$ where $D=(C\setminus R)\cup\{s\}$.  As $\Theta$ is an up-set, $D\in\Theta$.  As $R$ is $\Theta$-Cauchy, $\emptyset\neq R\cap D\subseteq\{s\}$ and hence $s\in R$, showing that $R$ is an up-set.
\end{proof}

Likewise, if $\Theta$ is directed, so are all $\Theta$-round $\Theta$-Cauchy up-sets.

\begin{prp}\label{Directed=>Directed}
If $R\subseteq\mathbb{P}$ is a $\Theta$-round $\Theta$-Cauchy up-set then
\[\Theta\text{ is directed}\qquad\Rightarrow\qquad R\text{ is directed}.\]
\end{prp}

\begin{proof}
Take any finite $F\subseteq R$.  As $R$ is $\Theta$-round, for all $f\in F$, we have $C_f\in\Theta$ with $C_f\cap R\leq f$.  As $\Theta$ is directed, we have $C\in\Theta$ with $C\leq C_f$, for all $f\in F$.  As $R$ is $\Theta$-Cauchy, we have $c\in C\cap R$.  For all $f\in F$, $C\leq C_f$ so we have $c_f\in C_f$ with $c_f\geq c$.  As $R$ is an up-set, $c\leq c_f\in C_f\cap R\leq f$, for all $f\in F$, showing that $R$ is directed.
\end{proof}

The same result for filters is immediate from \autoref{Upset=>Upset} and \autoref{Directed=>Directed}.

\begin{cor}\label{Filter=>Filter}
If $R\subseteq\mathbb{P}$ is $\Theta$-round and $\Theta$-Cauchy then
\[\Theta\text{ is a filter}\qquad\Rightarrow\qquad R\text{ is a filter}.\]
\end{cor}

As long as $\Theta$ avoids trivial subsets, so do $\Theta$-round up-sets.

\begin{prp}\label{Avoiding0}
If $\mathbb{P}$ has a minimum $0$ then
\[\emptyset\in\Theta\text{ or }\{0\}\in\Theta\qquad\Leftrightarrow\qquad\mathbb{P}\text{ is $\Theta$-round}\qquad\Leftrightarrow\qquad\exists\text{ $\Theta$-round up-set }R\ni0.\]
\end{prp}

\begin{proof}
If $\emptyset\in\Theta$ or $\{0\}\in\Theta$ then $\mathbb{P}$ is $\Theta$-round.  Conversely, if $\mathbb{P}$ is $\Theta$-round then, as $0\in\mathbb{P}$, we have some $C\in\Theta$ with $C\leq0$, the only possibilities for which are $C=\emptyset$ or $C=\{0\}$.  This proves the first $\Leftrightarrow$, while the second is immediate.
\end{proof}

Equivalent definitions of $\Theta$-round and $\Theta$-Cauchy could have been given in terms of the complement $\mathbb{P}\setminus R$.  For up-sets, these can be stated as follows.  Note here $\Theta^\leq$ denotes the $\leq$-closure of $\Theta$, i.e. the family of all subsets refined by subsets in $\Theta$.

\begin{prp}\label{CauchyRoundComplements}
If $R\subseteq\mathbb{P}$ is an up-set then
\begin{align}
\label{ThetaCauchyComplement}R\text{ is $\Theta$-Cauchy}\qquad&\Leftrightarrow\qquad\mathbb{P}\setminus R\notin\Theta^\leq.\\
\label{ThetaRoundComplement}R\text{ is $\Theta$-round}\qquad&\Leftrightarrow\qquad\forall r\in R\ ((\mathbb{P}\setminus R)\cup\{r\}\in\Theta^\leq).
\end{align}
\end{prp}

\begin{proof}\
\begin{itemize}
\item[\eqref{ThetaCauchyComplement}] By definition, $R$ is $\Theta$-Cauchy iff $\mathbb{P}\setminus R\notin\Theta^\subseteq$.  As $R$ is an up-set, $R$ is $\Theta$-Cauchy iff $R$ is $\Theta^\leq$-Cauchy iff $\mathbb{P}\setminus R\notin\Theta^{\leq\subseteq}=\Theta^\leq$.

\item[\eqref{ThetaRoundComplement}] By definition, $R$ is $\Theta$-round iff $(\mathbb{P}\setminus R)\cup r^\geq\in\Theta^\subseteq$,  for all $r\in R$.  As $R$ is an up-set, $\mathbb{P}\setminus R$ is a down-set so $(\mathbb{P}\setminus R)\cup r^\geq\in\Theta^\subseteq$ iff $(\mathbb{P}\setminus R)\cup\{r\}\in\Theta^\leq$.\qedhere
\end{itemize}
\end{proof}

Next, we have the following analog of \cite[Lemma 1 (b)]{Harris1971}.

\begin{prp}\label{CauchyUpsetRound}
If $Q\subseteq\mathbb{P}$ is a $\Theta$-Cauchy up-set, and $R\subseteq\mathbb{P}$ is $\Theta$-round then
\[Q\subseteq R\qquad\Rightarrow\qquad Q=R.\]
\end{prp}

\begin{proof}
For any $r\in R$, we have $C\in\Theta$ with $\emptyset\neq C\cap Q\subseteq C\cap R\leq r$, as $Q$ is $\Theta$-Cauchy, $Q\subseteq R$ and $R$ is $\Theta$-round.  Thus $r\in Q$, as $Q$ is an up-set.
\end{proof}

It follows immediately that any $\Theta$-round $\Theta$-Cauchy up-set $Q\subseteq\mathbb{P}$ is both a maximal $\Theta$-round subset and a minimal $\Theta$-Cauchy up-set.  In fact, among $\Theta$-Cauchy up-sets, `$\Theta$-round' really just means `minimal'.

\begin{prp}\label{Round<=>Minimal}
If $R\subseteq\mathbb{P}$ is a $\Theta$-Cauchy up-set then
\[R\text{ is $\Theta$-round}\qquad\Leftrightarrow\qquad R\text{ is minimal}.\]
\end{prp}

\begin{proof}
The $\Rightarrow$ part follows immediately from \autoref{CauchyUpsetRound}.  Conversely, say $R$ is not $\Theta$-round, which means we have some $r\in R$ such that $C\cap R\not\leq r$, for all $C\in\Theta$.  But this means $R\setminus r^\geq\subsetneqq R$ is a $\Theta$-Cauchy up-set so $R$ is not minimal.
\end{proof}

If we expand $\Theta$ to $\Theta^\leq$ then `round' even means minimal among arbitrary $\Theta^\leq$-Cauchy subsets.

\begin{thm}\label{RCU<=>MleqC}
For any $R\subseteq\mathbb{P}$,
\[R\text{ is a $\Theta$-round $\Theta$-Cauchy up-set}\qquad\Leftrightarrow\qquad R\text{ is minimal $\Theta^\leq$-Cauchy}.\]
\end{thm}

\begin{proof}
Say $R$ is $\Theta$-Cauchy and $\leq$-closed.  For any $D\in\Theta^\leq$, we have $C\in\Theta$ with $C\leq D$ and hence, as $R$ is $\leq$-closed, $\emptyset\neq R\cap C\leq R\cap D$.  As $R$ is $\leq$-closed, this means $R\cap D\neq\emptyset$, i.e. $R$ is $\Theta^\leq$-Cauchy.  If $R$ is also round then, for any $r\in R$, we have $C\in\Theta$ with $C\cap R\leq r$.  Setting $D=(C\setminus R)\cup\{r\}$, this means $C\leq D$ and hence $D\in\Theta^\leq$.  By definition, $D\cap(R\setminus\{r\})=\emptyset$ which means $R\setminus\{r\}$ is not $\Theta^\leq$-Cauchy.  As $r\in R$ was arbitrary, this means $R$ is minimal $\Theta^\leq$-Cauchy.

Say $R$ is a minimal $\Theta^\leq$-Cauchy subset, i.e.
\[\tag{Minimality}\forall r\in R\ \exists C\in\Theta^\leq\ (C\cap R=\{r\}).\]
Thus $R$ is certainly $\Theta^\leq$-round and hence an up-set, by \autoref{Upset=>Upset}.  To see that $R$ is also $\Theta$-round, take any $r\in R$ and $C\in\Theta^\leq$ with $C\cap R=\{r\}$.  Further take $D\in\Theta$ with $D\leq C$.  If we had $d\in D\cap R$ with $d\nleq r$ then we would have $c\in C$ with $d\leq c$ and hence $c\in R$, as $R$ is an up-set.  But then $c\nleq r$ too so $c\neq r$, contradicting $C\cap R=\{r\}$.  Thus we must have had $D\cap R\leq r$, i.e. $R$ is $\Theta$-round.
\end{proof}

Note that being $\Theta$-Cauchy depends only on $\Theta$, not on the order structure.  So what the above result shows is that, once we know $\Theta^\leq$, we can forget about the order and yet still determine the $\Theta$-round $\Theta$-Cauchy up-sets.  Sometimes $\Theta$ even determines the order structure itself \textendash\, see \autoref{SubbasisOrder} below.

\section{The Spectrum}\label{TheSpectrum}

For any $S\subseteq\mathbb{P}$, let
\[\Theta_S=\{Q\in\Theta:S\subseteq Q\}.\]
Here we again identify elements with singleton subsets, i.e. $\Theta_p=\Theta_{\{p\}}$.

\begin{dfn}
The \emph{$\Theta$-spectrum} is the space with subbasis $(\widehat{\Theta}_p)_{p\in\mathbb{P}}$ and points
\[\widehat{\Theta}=\{R\subseteq\mathbb{P}:R\text{ is a $\Theta$-round $\Theta$-Cauchy up-set}\}.\]
\end{dfn}

The first natural question to ask is what kind of spaces arise in this way.

\begin{dfn}
We call $\mathbb{P}\subseteq\mathcal{P}(X)$ a \emph{$T_1$ family} if, for all $x,y\in X$,
\[\tag{$T_1$}x\neq y\qquad\Leftrightarrow\qquad\exists O,N\in\mathbb{P}\ (x\in O\setminus N\text{ and }y\in N\setminus O).\]
\end{dfn}

Recall that $X$ is said to be a \emph{$T_1$ space} iff $\mathcal{O}(X)$ is a $T_1$ family.  In this case, one immediately sees that every subbasis $\mathbb{P}\subseteq\mathcal{O}(X)$ is also a $T_1$ family.

\begin{prp}\label{T1}
The $\Theta$-spectrum is always a $T_1$ space.
\end{prp}

\begin{proof}
For any distinct $R,S\subseteq\mathbb{P}$ that are maximal or minimal within any $\Theta\subseteq\mathcal{P}(\mathbb{P})$, we can find $r\in R\setminus S$ and $s\in S\setminus R$.  But this means $R\in\Theta_r\not\ni S$ and $S\in\Theta_s\not\ni R$, which is precisely the $T_1$ condition.  This applies to $\widehat{\Theta}$ because, by \autoref{CauchyUpsetRound}, \autoref{Round<=>Minimal} and \autoref{RCU<=>MleqC},
\begin{align*}
\widehat{\Theta}&=\{R\subseteq\mathbb{P}:R\text{ is a minimal $\Theta^\leq$-Cauchy subset}\}\\
&=\{R\subseteq\mathbb{P}:R\text{ is a minimal $\Theta$-Cauchy up-set}\}\\
&\subseteq\{R\subseteq\mathbb{P}:R\text{ is a maximal $\Theta$-Round up-set}\}.\qedhere
\end{align*}
\end{proof}

Conversely, all $T_1$ spaces arise in this way.  First let
\[\mathcal{C}_X(\mathbb{P})=\{C\subseteq\mathbb{P}:X\subseteq\bigcup C\},\]
i.e. $\mathcal{C}_X(\mathbb{P})$ denotes the family of all $\mathbb{P}$-covers of $X$.

\begin{prp}\label{CauchyCovers}
If $\mathbb{P}\subseteq\mathcal{P}(X)$ then $\Theta\subseteq\mathcal{C}_X(\mathbb{P})$ if and only if, for all $x\in X$,
\[\mathbb{P}_x=\{N\in\mathbb{P}:x\in N\}\text{ is $\Theta$-Cauchy}.\]
\end{prp}

\begin{proof}
If $\Theta\subseteq\mathcal{C}_X(\mathbb{P})$ then each $C\in\Theta$ satisfies $\bigcup C=X$ so, for each $x\in X$, we have $c\in C$ with $x\in c$ and hence $c\in\mathbb{P}_x$, i.e. $\mathbb{P}_x$ is $\Theta$-Cauchy.  If $\Theta\nsubseteq\mathcal{C}_X(\mathbb{P})$ then we have some $C\in\Theta$ and $x\in X$ with $x\notin\bigcup C$ and hence $C\cap\mathbb{P}_x=\emptyset$, i.e. $\mathbb{P}_x$ is not $\Theta$-Cauchy.
\end{proof}

\begin{prp}\label{T1Recovery}
If $X$ is a $T_1$ space, $\leq\ =\ \subseteq$ on a subbasis $\mathbb{P}\subseteq\mathcal{O}(X)$ and $\Theta$ is a coinitial subset of $\mathcal{C}_X(\mathbb{P})$ then $X$ is homeomorphic to $\widehat{\Theta}$ via the map
\[x\mapsto\mathbb{P}_x=\{N\in\mathbb{P}:x\in N\}.\]
\end{prp}

\begin{proof}
For every $x\in\mathbb{P}$, $\mathbb{P}_x$ is $\Theta^\leq$-Cauchy, by \autoref{CauchyCovers}, as $\Theta^\leq\subseteq\mathcal{C}_X(\mathbb{P})$.  Also observe that, for all distinct $x,y\in X$, we have $O\in\mathbb{P}$ with $x\notin O\ni y$, as $X$ is $T_1$ and $\mathbb{P}$ is a subbasis.  In particular, $\mathbb{P}_x\neq\mathbb{P}_y$ and hence $x\mapsto\mathbb{P}_x$ is injective.  This observation also means that $\mathbb{P}\setminus\mathbb{P}_x$ always covers $X\setminus\{x\}$ so, for each $N\in\mathbb{P}_x$,
\[\{N\}\cup(\mathbb{P}\setminus\mathbb{P}_x)\in\mathcal{C}_X(\mathbb{P})=\Theta^\leq.\]
This shows that $\mathbb{P}_x$ is minimal $\Theta^\leq$-Cauchy and hence a $\Theta$-round $\Theta$-Cauchy up-set, by \autoref{RCU<=>MleqC}.  Thus $x\mapsto\mathbb{P}_x$ maps $X$ to $\widehat{\Theta}$.  Moreover, the image is homeomorphic to $X$ because $\mathbb{P}$ and $(\widehat\Theta_p)_{p\in\mathbb{P}}$ are subbases of $X$ and $\widehat\Theta$ and
\[x\in p\qquad\Leftrightarrow\qquad p\in\mathbb{P}_x\qquad\Leftrightarrow\qquad\mathbb{P}_x\in\widehat\Theta_p.\]

On the other hand, any $\Theta$-Cauchy up-set $Q\subseteq\mathbb{P}$ must contain $\mathbb{P}_x$, for some $x\in X$.  Otherwise, $\mathbb{P}\setminus Q$ would have non-empty intersection with each $\mathbb{P}_x$ and hence cover $X$, i.e. $\mathbb{P}\setminus Q\in\mathcal{C}_X(\mathbb{P})=\Theta^\leq$.  So we would have $C\in\Theta$ with $C\leq\mathbb{P}\setminus Q$ which, as $Q$ is $\Theta$-Cauchy, would yield $c\in C\cap Q$ and $p\in\mathbb{P}\setminus Q$ with $c\leq q$, contradicting the fact that $Q$ is an up-set.  As we already know $\mathbb{P}_x\in\widehat{\Theta}$, for each $x\in X$, $(\mathbb{P}_x)_{x\in X}$ must exhaust all minimal/$\Theta$-round $\Theta$-Cauchy up-sets, i.e. $\widehat{\Theta}=\{\mathbb{P}_x:x\in X\}$.
\end{proof}

So \autoref{T1} tells us that $\Theta$ can always be represented concretely on some $T_1$ space.  Conversely, \autoref{T1Recovery} tells us that if $\Theta$ already is a coinitial family of covers of a $T_1$ space then we can recover that space via the spectrum.  This could already be seen as a weak kind of duality between nearness posets and $T_1$ spaces.  But there are a couple of things that stop it from really being a duality.

Firstly, given some abstract $\Theta$, there is no guarantee that resulting covers on $\widehat\Theta$ will be coinitial.  Secondly, there is also no guarantee that $\Theta$ will be faithful or, more precisely, that $\leq$ on $\mathbb{P}$ will be faithfully represented as $\subseteq$ on $(\widehat\Theta_p)_{p\in\mathbb{P}}$, i.e.
\[\label{Faithful}\tag{Faithful}p\leq q\qquad\Leftrightarrow\qquad\widehat\Theta_p\subseteq\widehat\Theta_q.\]
To get a true duality we need some extra `admissibility' condition on $\Theta$.  Several such conditions will be discussed in the following sections.

For the moment we just note that the $\Rightarrow$ part of \eqref{Faithful} always holds, even when $\leq$ is replaced by a weaker preorder defined from $\Theta$.  First let
\[\Theta^S=\{R\subseteq\mathbb{P}:R\cup S\in\Theta\},\]
for any $S\subseteq\mathbb{P}$.  Then define $\leq_\Theta$ on $\mathbb{P}$ by
\[\tag{$\Theta$-Preorder}p\leq_\Theta q\qquad\Leftrightarrow\qquad\Theta^p\subseteq\Theta^q.\]

\begin{prp}\label{SubbasisOrder}
If $\mathbb{P}\subseteq\mathcal{P}(X)$ is a $T_1$ family and $\Theta=\mathcal{C}_X(\mathbb{P})$ then
\[\tag{$\subseteq\ =\ \leq_\Theta$}p\subseteq q\qquad\Leftrightarrow\qquad p\leq_\Theta q.\]
\end{prp}

\begin{proof}
Certainly if $p\subseteq q$ then, for any $S\subseteq\mathbb{P}$, $p\cup\bigcup S=X$ implies $q\cup\bigcup S=X$, i.e. $p\leq_\Theta q$.  Conversely, if $p\nsubseteq q$ then we have some $x\in p\setminus q$.  As $\mathbb{P}$ is $T_1$, for every $y\in X\setminus\{x\}$, we have $r_y\in\mathbb{P}$ with $x\notin r_y\ni y$.  So $p\cup\bigcup_{y\in X\setminus\{x\}}r_y=X$ but $x\notin q\cup\bigcup_{y\in X\setminus\{x\}}r_y$ and hence $p\nleq_\Theta q$, as required.
\end{proof}

Alternatively, one could note that $\Theta$ above is faithful, by \autoref{T1Recovery}, and then the following abstract result yields \autoref{SubbasisOrder}.

\begin{prp}\label{FaithfulRightarrow}
For any $p,q\in\mathbb{P}$,
\[p\leq q\qquad\Rightarrow\qquad p\leq_{\Theta^\leq}q\qquad\Rightarrow\qquad\widehat\Theta_p\subseteq\widehat\Theta_q.\]
\end{prp}

\begin{proof}
For the first implication, note that if $p\leq q$ and $S\in\Theta^{\leq p}$, i.e. $\{p\}\cup S\in\Theta^\leq$ then $\{q\}\cup S\in\Theta^\leq$, i.e. $S\in\Theta^{\leq p}$, showing that $p\leq_{\Theta^\leq}q$.  For the next implication, say $p\leq_{\Theta^\leq}q$ and take $R\in\widehat\Theta_p$.  As $R$ is $\Theta$-round, we have some $C\in\Theta$ with $R\cap C\leq p$.  Let $D=(C\setminus p^\geq)\cup\{q\}$ so $C\leq D$ and hence $D\in\Theta^\leq$.  Thus $R\cap D=\{q\}$, as $R$ is $\Theta^\leq$-Cauchy and $R\cap C\leq p$, so $R\in\widehat\Theta_q$, as required.
\end{proof}

For $\Theta$ to be faithful, it follows that we must at least have $\leq\ =\ \leq_{\Theta^\leq}$.  There are a couple of easily identifiable situations in which this fails very badly.  For example, if $\Theta=\emptyset$ then $\Theta^\leq=\Theta^{\leq p}=\emptyset$, for all $p\in\mathbb{P}$, while if $\emptyset\in\Theta$ then $\Theta^\leq=\Theta^{\leq p}=\mathcal{P}(\mathbb{P})$, for all $p\in\mathbb{P}$.  In either case, $\leq_{\Theta^\leq}$ is the weakest relation possible, i.e.
\[\leq_\emptyset\ =\ \leq_{\mathcal{P}(\mathbb{P})}\ =\ \mathbb{P}\times\mathbb{P}.\]
It then follows from \autoref{FaithfulRightarrow} that the representation in \autoref{T1} will be completely degenerate.  This can also be shown directly as follows.

\begin{prp}
We always have the following general implications.
\begin{align}
\label{Empty=Theta}\emptyset=\Theta\qquad&\Leftrightarrow\qquad\emptyset\in\widehat\Theta\qquad\Leftrightarrow\qquad\{\emptyset\}=\widehat\Theta.\\
\label{EmptyinTheta}\emptyset\in\Theta\qquad&\Rightarrow\qquad\emptyset=\widehat\Theta.\\
\intertext{If $\mathbb{P}$ has a minimum $0$ then}
\label{0inTheta}\{0\}\in\Theta\text{ and }\emptyset\notin\Theta\qquad&\Leftrightarrow\qquad\mathbb{P}\in\widehat\Theta\qquad\Leftrightarrow\qquad\{\mathbb{P}\}=\widehat\Theta.\\
\label{0notinTheta}\{0\}\notin\Theta\qquad&\Rightarrow\qquad\widehat\Theta\subseteq\mathcal{P}(\mathbb{P}\setminus\{0\}).
\end{align}
\end{prp}

\begin{proof}\
\begin{itemize}
\item[\eqref{Empty=Theta}] If $\emptyset=\Theta$ then any $R\subseteq\mathbb{P}$ is vacuously $\Theta$-Cauchy.  However $\emptyset$ is the only (again vacuously) $\Theta$-round subset, i.e. $\{\emptyset\}=\Theta$.  Conversely, if $\emptyset\in\widehat\Theta$ then $\emptyset$ is $\Theta$-Cauchy, which is only possible if $\Theta=\emptyset$.
\item[\eqref{EmptyinTheta}] If $\emptyset\in\Theta$ then $\mathbb{P}$ has no $\Theta$-Cauchy subsets.
\item[\eqref{0inTheta}] If $\{0\}\in\Theta$ then $\mathbb{P}$ is a $\Theta$-round up-set and also the only possible $\Theta$-Cauchy up-set.  If $\emptyset\notin\Theta$ then $\mathbb{P}$ is indeed $\Theta$-Cauchy and hence $\widehat\Theta=\{\mathbb{P}\}$.  Conversely, if $\mathbb{P}\in\widehat\Theta$ then $\mathbb{P}$ is $\Theta$-round so $\emptyset\in\Theta$ or $\{0\}\in\Theta$, by \autoref{Avoiding0}.  However, $\mathbb{P}$ is also $\Theta$-Cauchy so $\emptyset\notin\Theta$ and hence $\{0\}\in\Theta$.
\item[\eqref{0notinTheta}] If $\emptyset\in\Theta$ then $\widehat\Theta=\emptyset\subseteq\mathcal{P}(\mathbb{P}\setminus\{0\})$.  If $\emptyset\notin\Theta$ and $\{0\}\notin\Theta$ then $\widehat\Theta\subseteq\mathcal{P}(\mathbb{P}\setminus\{0\})$, by \autoref{Avoiding0}.\qedhere
\end{itemize}
\end{proof}

In general, the implication in \eqref{EmptyinTheta} can not be reversed.  Moreover, even when $\mathbb{P}$ contains minimal $\Theta$-Cauchy filters, there may still be no minimal $\Theta$-Cauchy up-sets, i.e. $\widehat\Theta=\emptyset$, as in the following example.

\begin{xpl}\label{FirstExample}
Define $\leq$ on $\mathbb{P}=\mathbb{N}\times\mathbb{N}$ by
\[(m,n)\leq(m',n')\quad\Leftrightarrow\quad m\geq m'\text{ and }(m+n=m'+n'\text{ or }(n=1\text{ and }m\geq m'+n')).\]
The Hasse diagram of $\mathbb{P}$ is as follows.
\begin{center}
\begin{tikzpicture}
    \node (0-1) at (-1,0) [align=center]{$R_1$};
    \node (1-1) at (-1,-1) [align=center]{$R_2$};
    \node (2-1) at (-1,-2) [align=center]{$R_3$};
		
    \node (00) at (0,0) [align=center]{$(1,1)$};
		\node (10) at (0,-1) [align=center]{$(2,1)$};
		\node (20) at (0,-2) [align=center]{$\quad\ (3,1)\ldots$};
		\node (01) at (1,0) [align=center]{$(1,2)$};
		\node (11) at (1,-1) [align=center]{$\quad\ (2,2)\ldots$};
		\node (02) at (2,0) [align=center]{$\quad\ (1,3)\ldots$};
		\draw (00)--(10)--(20);
		\draw (01)--(10);
		\draw (02)--(11)--(20);
		\draw[dotted] (20)--(0,-3);
\end{tikzpicture}
\end{center}
This makes $\mathbb{P}$ a (reverse) graded poset where the rank is determined by the first coordinate.  So we can take $\Theta=\{R_m:m\in\mathbb{N}\}$ to be the decreasing rank levels
\[R_m=\{(m,n):n\in\mathbb{N}\}\]

We claim that $\mathbb{P}$ itself is the unique $\Theta$-Cauchy filter.  To see this, first note that $\mathbb{P}$ is a $\wedge$-semilattice where
\[(m,n)\wedge(m',n')=\begin{cases}(m\vee m',n\wedge n')&\text{if }m+n=m'+n'.\\ (m+n-1,1)&\text{if }m+n\geq m'+n'.\end{cases}\]
Given a $\Theta$-Cauchy filter $Q\subseteq\mathbb{P}$, take any $(m,n)\in Q$.  As $Q$ is $\Theta$-Cauchy, for any $l\geq m+n$, we have $(k,l)\in Q$, for some $k\in\mathbb{N}$.  As $Q$ is a filter,
\[(k+l+1,1)=(m,n)\wedge(k,l)\in Q.\]
Thus, for arbitrarily large $j\in\mathbb{N}$, we have $(j,1)\in Q$ and hence $(j,1)^\leq\subseteq Q$, again because $Q$ is a filter.  Thus $\mathbb{P}=\bigcup_{j\in\mathbb{N}}(j,1)^\leq=R$, proving the claim.

However, $\mathbb{P}$ is not $\Theta$-round.  Indeed, for any $(m,n)\in\mathbb{P}$, $(m',n+1)\nleq(m,n)$ and hence $R_{m'}\nleq(m,n)$, for all $m'\in\mathbb{N}$.  Thus $\mathbb{P}$ is not minimal among $\Theta$-Cauchy up-sets, by \autoref{Round<=>Minimal}.  Indeed, $\mathbb{P}$ has no minimal/$\Theta$-round $\Theta$-Cauchy up-sets, as these would be filters, by \autoref{Directed=>Directed}, i.e. $\widehat{\Theta}$ is empty.
\end{xpl}

It was crucial in this example that $\Theta$ consisted of infinite subsets.  Indeed, if $\emptyset\notin\Theta\subseteq\mathcal{F}(\mathbb{P})$ then taking $C=\emptyset$ in \eqref{CoinitialCovers} below yields $\widehat{\Theta}\neq\emptyset$.

\section{Wallman Duality}\label{WallmanDuality}

We can now reformulate Wallman's classic duality from \cite{Wallman1938}.

\begin{thm}\label{ThetaCompact}
If $\Theta\subseteq\mathcal{F}(\mathbb{P})$ then $\widehat{\Theta}$ is compact.  Moreover, for $p,q\in\mathbb{P}$ and $C\subseteq\mathbb{P}$,
\begin{align}
\label{ThetaPreorder=subseteq}\widehat\Theta_p\subseteq\widehat\Theta_q\qquad&\Leftrightarrow\qquad p\leq_{\Theta^\leq}q.\\
\label{CoinitialCovers}\widehat\Theta\subseteq\bigcup_{c\in C}\widehat{\Theta}_c\qquad&\Leftrightarrow\qquad C\in\Theta^\leq.
\end{align}
\end{thm}

\begin{proof}
First note that, for compactness, it suffices to prove \eqref{CoinitialCovers}.  Indeed, \eqref{CoinitialCovers} implies that the covers of $\widehat{\Theta}$ of the form $(\widehat{\Theta}_c)_{c\in C}$, for $C\in\Theta^\leq\cap\mathcal{F}(\mathbb{P})$, are $\subseteq$-coinitial in the collection of all covers consisting of open sets from $(\widehat{\Theta}_p)_{p\in\mathbb{P}}$.   In other words, every subbasic open cover has a finite subcover and hence $\widehat{\Theta}$ is compact, by the Alexander subbasis theorem.

\begin{itemize}
\item[\eqref{CoinitialCovers}]  For $\Leftarrow$, take $C\in\Theta^\leq$, i.e. $D\leq C$, for some $D\in\Theta$.  As $R\cap D\neq\emptyset$, for any $\Theta$-Cauchy $R$, it follows $R\cap C\neq\emptyset$, for any $\Theta$-Cauchy up-set $R$.

Conversely, say $C\notin\Theta^\leq$.  This means $D\setminus C^\geq\neq\emptyset$, for all $D\in\Theta$, and hence $\mathbb{P}\setminus C^\geq$ is a $\Theta$-Cauchy up-set.  We claim that we can then apply the Kuratowski-Zorn lemma to obtain a minimal $\Theta$-Cauchy up-set $R\subseteq\mathbb{P}\setminus C^\geq$.  To see this, say we have a chain, or even a $\subseteq$-directed collection $\Phi$ of $\Theta$-Cauchy up-sets.  Certainly $\bigcap\Phi$ will also be an up-set.  Also, as $\Phi$ is $\subseteq$-directed and each $C\in\Theta$ is finite, $C\cap\bigcap\Phi=\emptyset$ would imply $C\cap Q=\emptyset$, for some $Q\in\Phi$, contradicting the fact $Q$ is $\Theta$-Cauchy.  This proves the claim, which yields $R\in\widehat{\Theta}$ with $R\cap C=\emptyset$ and hence $R\notin\bigcup_{c\in C}\widehat{\Theta}_c$.  So $R$ witnesses $\bigcup_{c\in C}\widehat{\Theta}_c\neq\widehat{\Theta}$, as required.

\item[\eqref{ThetaPreorder=subseteq}]  The $\Leftarrow$ part appears in \autoref{FaithfulRightarrow}.  Conversely, say $p\nleq_{\Theta^\leq}q$, so we have some $S\subseteq\mathbb{P}$ with $S\cup\{p\}\in\Theta^\leq$ but $S\cup\{q\}\notin\Theta^\leq$.  Then \eqref{CoinitialCovers} yields $R\in\widehat\Theta$ disjoint from $D\cup\{q\}$.  As $R$ is a $\Theta^\leq$-Cauchy and $D\cup\{p\}\in\Theta^\leq$, it follows that $p\in R$, i.e. $R\in\widehat\Theta_p\setminus\widehat\Theta_q$. \qedhere
\end{itemize}
\end{proof}

\begin{rmk}
The careful reader might note that we used the Alexander subbasis theorem from \cite{Alexander1939} in the proof of \autoref{ThetaCompact}, even though this would not have been available to Wallman in \cite{Wallman1938}.  However, examination of \cite{Wallman1938} reveals that Wallman essentially proved the same theorem in order to derive his duality.  Alexander even cited \cite{Wallman1938} in \cite{Alexander1939}, so `Alexander-Wallman subbasis theorem' might be more historically accurate.
\end{rmk}

So any $\Theta\subseteq\mathcal{F}(\mathbb{P})$ can be concretely represented as a coinitial collection of subbasic covers of the compact $T_1$ space $\widehat\Theta$ (see \autoref{T1}).  Moreover, $\leq$ on $\mathbb{P}$ is faithfully represented as $\subseteq$ on $(\widehat\Theta_p)_{p\in\mathbb{P}}$ if (and only if) $\leq\ =\ \leq_{\Theta^\leq}$.  Conversely, if $\Theta$ is a coinitial collection of finite subbasic covers of a compact $T_1$ space then taking $\leq\ =\ \subseteq$ on the subbasis $\mathbb{P}$, we have $\leq\ =\ \leq_{\Theta^\leq}$, by \autoref{SubbasisOrder}.  Moreover, we then recover the original space as $\widehat\Theta$, by \autoref{T1Recovery}.

This shows that we do indeed have a duality between appropriate nearness posets and compact $T_1$ spaces or, more precisely, certain covers of these spaces.  We can express this duality in a slightly more convenient form by considering all finite covers of the space rather than their coinitial subfamilies.  In this case, $\leq\ =\ \leq_{\Theta^\leq}$ corresponds to the following `admissibility' condition on $\Theta$.
\[\label{WallmanAdmissible}\tag{Wallman Admissible}p\leq q\qquad\Leftrightarrow\qquad\forall C\in\Theta\ ((C\setminus\{p\})\cup\{q\}\in\Theta).\]

Let us denote the `finite extension' relation by $\subseteq_\mathcal{F}$, i.e.
\[R\subseteq_\mathcal{F}S\qquad\Leftrightarrow\qquad R\subseteq S\text{ and }S\setminus R\text{ is finite}.\]

\begin{prp}\label{WallmanAdmissibleEquivalent}
$\Theta$ is Wallman admissible iff $\Theta$ is $\subseteq_\mathcal{F}$-closed and $\leq\ =\ \leq_\Theta$.
\end{prp}

\begin{proof}
If $\Theta$ is Wallman admissible then, for any $C\in\Theta$, the definition with $p=q$ yields $C\cup\{p\}=(C\setminus\{p\})\cup\{p\}\in\Theta$.  From this it follows that $\Theta$ must be $\subseteq_\mathcal{F}$-closed.  Thus it suffices to show that, when $\Theta$ is $\subseteq_\mathcal{F}$-closed,
\[p\leq_\Theta q\qquad\Leftrightarrow\qquad\forall C\in\Theta\ ((C\setminus\{p\})\cup\{q\}\in\Theta).\]
To see this, say $p$ and $q$ satisfy the right hand side and take $S\subseteq\mathbb{P}$ with $S\cup\{p\}\in\Theta$.  If $p\notin S$ then it follows that $S\cup\{q\}=((S\cup\{p\})\setminus\{p\})\cup\{q\}\in\Theta$.  While if $p\in S$ then $S\cup\{q\}\supseteq S=S\cup\{p\}\in\Theta$ so again $S\cup\{q\}\in\Theta$, as $\Theta$ is $\subseteq_\mathcal{F}$-closed, i.e. $p\leq_\Theta q$.  Conversely, if $p\leq_\Theta q$ then, for any $C\in\Theta$, $C\subseteq C\cup\{p\}=(C\setminus\{p\})\cup\{p\}\in\Theta$, as $\Theta$ is $\subseteq_\mathcal{F}$-closed, and hence $(C\setminus\{p\})\cup\{q\}\in\Theta$.
\end{proof}

\begin{prp}\label{SubsetClosedTheta}
If $\Theta$ is $\subseteq$-closed in $\mathcal{F}(\mathbb{P})$ then $\Theta$ is $\leq_\Theta$-closed in $\mathcal{F}(\mathbb{P})$.
\end{prp}

\begin{proof}
Take $F,G\in\mathcal{F}(\mathbb{P})$ with $F\leq_\Theta G$.  Thus, for any $f\in F$, we have $g\in G$ with $f\leq_\Theta g$.  As $(F\setminus\{f\})\cup\{f\}=F\in\Theta$, this means $(F\setminus\{f\})\cup\{g\}\in\Theta$.  Taking $f'\in F\setminus\{f\}$, we have $g'\in G$ with $f'\leq_\Theta g'$ which again yields $(F\setminus\{f,f'\})\cup\{g,g'\}\in\Theta$.  As $F$ is finite, we can continue in this way to get $G'\in\Theta$, for some $G'\subseteq G$.  As $\Theta$ is $\subseteq$-closed in $\mathcal{F}(\mathbb{P})$, this means $G\in\Theta$, as required.
\end{proof}

From \autoref{WallmanAdmissibleEquivalent} and \autoref{SubsetClosedTheta}, we immediately have the following.

\begin{cor}\label{Wallman}
The following are equivalent when $\Theta\subseteq\mathcal{F}(\mathbb{P})$.
\begin{enumerate}
\item $\Theta$ is Wallman admissible.
\item $\Theta$ is $\subseteq$-closed in $\mathcal{F}(\mathbb{P})$ and $\leq\ =\ \leq_\Theta$.
\item $\Theta$ is $\leq$-closed in $\mathcal{F}(\mathbb{P})$ and $\leq\ =\ \leq_\Theta$.
\end{enumerate}
\end{cor}

To reformulate Wallman duality, let us call $(\mathbb{P},\leq,\Theta)$ a \emph{Wallman poset} if $\Theta\subseteq\mathcal{F}(\mathbb{P})$ is Wallman admissible.  Roughly speaking, we can state the result as follows.
\begin{thm}\label{WallmanDualitySummary}
\[\textbf{\upshape Wallman posets are dual to subbases of compact $T_1$ spaces}.\]
\end{thm}
More precisely, any Wallman poset $(\mathbb{P},\leq,\Theta)$ can be concretely represented as the subbasis $(\widehat\Theta_p)_{p\in\mathbb{P}}$ of the compact $T_1$ space $\widehat\Theta$ where $\leq$ becomes $\subseteq$ and $\Theta$ becomes the family of all finite covers $(\{\widehat\Theta_c:c\in C\})_{C\in\Theta}$ of $\widehat\Theta$, by \autoref{T1}, \autoref{ThetaCompact} and \autoref{WallmanAdmissibleEquivalent}.  On the other hand, if $\mathbb{P}$ is a subbasis of a $T_1$ space $X$ and $\Theta=\mathcal{C}_X(\mathbb{P})\cap\mathcal{F}(\mathbb{P})$ then $(\mathbb{P},\subseteq,\Theta)$ is a Wallman poset, by \autoref{SubbasisOrder}.  Moreover, in this case $\widehat\Theta$ is homeomorphic to the original space $X$, by \autoref{T1Recovery}.

Lastly, let us note that the order structure in Wallman duality can reduce to mere equality.  So it really is the covers that are playing the pivotal role.

\begin{xpl}\label{Order=Equality}
Let $X=[0,1]$ and let $\mathbb{P}\subseteq\mathcal{O}(X)$ consist of all sets of the form $(a,b)\cup(c,d)$ where $a\leq b<c\leq d$ and $(b-a)+(d-c)=\frac{1}{2}$.  This last condition ensures that distinct elements of $\mathbb{P}$ are incomparable, i.e. $p\subseteq q$ or $q\subseteq p$ implies $p=q$.  But it is not hard to see that any open interval $(e,f)$ with $f-e\leq\frac{1}{2}$ is of the form $p\cap q$, for some $p,q\in\mathbb{P}$.  Thus $\mathbb{P}$ is a subbasis for $X$.
\end{xpl}

\section{Near Subsets}\label{NearSubsets}

To go any further, we must first discuss near subsets.

\begin{dfn}\label{NearDefinition}
We call $S\subseteq\mathbb{P}$ \emph{$\Theta$-near} if
\[\tag{$\Theta$-Near}\exists D\notin\Theta\ \forall s\in S\ (D\cup\{s\}\in\Theta).\]
\end{dfn}

If $S$ is $\Theta$-near then we immediately see that any $Q\subseteq S$ is again $\Theta$-near.  Likewise,
\begin{equation}\label{BoundedNear}
Q\geq S\quad\text{and}\quad S\text{ is $\Theta^\leq$-near}\qquad\Rightarrow\qquad Q\text{ is $\Theta^\leq$-near}.
\end{equation}
In particular, this applies to points of the spectrum.

\begin{prp}\label{xNear}
Every $R\in\widehat\Theta$ is maximal $\Theta^\leq$-near.
\end{prp}

\begin{proof}
By \autoref{CauchyRoundComplements}, $\mathbb{P}\setminus R\notin\Theta^\leq$ and $(\mathbb{P}\setminus R)\cup\{r\}\in\Theta^\leq$, for all $r\in R$, so $R$ is $\Theta^\leq$-near.  If $R\cup\{p\}$ were $\Theta^\leq$-near, for some $p\in\mathbb{P}\setminus R$, then we would have $D\notin\Theta^\leq$ with $D\cup\{r\}\in\Theta^\leq$, for all $r\in R\cup\{p\}$.  Thus $D\cap(R\cup\{p\})=\emptyset$ and hence $D\cup\{p\}\subseteq\mathbb{P}\setminus R\notin\Theta^\leq$, a contradiction, proving maximality.
\end{proof}

When dealing with arbitrary covers, `near' just means `non-empty intersection'.

\begin{prp}\label{NearForAllCovers}
If $\mathbb{P}\subseteq\mathcal{P}(X)$ is a $T_1$ family and $\Theta=\mathcal{C}_X(\mathbb{P})$ then, for all $S\subseteq\mathbb{P}$,
\[\bigcap S\neq\emptyset\qquad\Leftrightarrow\qquad S\text{ is $\Theta$-near}\]
\end{prp}

\begin{proof}
For $\Rightarrow$, note that $\bigcap S\neq\emptyset$ means $S\subseteq\mathbb{P}_x$, for some $x\in S$.  As $\mathbb{P}$ is $T_1$, $\mathbb{P}_x\in\widehat\Theta$ so $\mathbb{P}_x$ is $\Theta$-near, by \autoref{xNear}, and hence $S\subseteq\mathbb{P}_x$ is also $\Theta$-near.  Conversely, if $S$ is $\Theta$-near then we have $D\notin\mathcal{C}_X(\mathbb{P})$ with $D\cup\{s\}\in\mathcal{C}_X(\mathbb{P})$, for all $s\in S$.  Thus $\bigcup D\neq X$ while $s\cup\bigcup D=X$, for all $s\in S$, so $\emptyset\neq X\setminus\bigcup D\subseteq\bigcap S$.
\end{proof}

In particular, in the above situation every singleton $S=\{p\}$ is $\Theta$-near, except when $p=\emptyset$.  In general, we consider this as a weak kind of admissibility condition.

\begin{dfn}
We call $\Theta$ \emph{weakly admissible} if $\{p\}$ is $\Theta$-near whenever $p^\leq\neq\mathbb{P}$.
\end{dfn}

\begin{prp}\label{Wallman=>Weakly}
If $\Theta$ is Wallman admissible then $\Theta$ is weakly admissible.
\end{prp}

\begin{proof}
Say $\Theta$ is Wallman admissible and take $p^\leq\neq\mathbb{P}$.  So we have some $q\ngeq p$ and Wallman admissibility yields $C\in\Theta$ such that $(C\setminus\{p\})\cup\{q\}\notin\Theta$.  As $C\setminus\{p\}\subseteq_\mathcal{F}(C\setminus\{p\})\cup\{q\}$, \autoref{WallmanAdmissibleEquivalent} shows that $C\setminus\{p\}\notin\Theta$ and hence $\{p\}$ is $\Theta$-near.
\end{proof}

When $\mathbb{P}$ has a minimum $0$, weak admissibility specifically does not require that $\{0\}$ is $\Theta^\leq$-near.  Indeed, $\{0\}$ is only $\Theta^\leq$-near in the somewhat trivial case of \eqref{0inTheta}.

\begin{prp}\label{0near}
If $\mathbb{P}$ has a minimum $0$ then
\[\{0\}\text{ is $\Theta^\leq$-near}\qquad\Leftrightarrow\qquad\{0\}\in\Theta\text{ and }\emptyset\notin\Theta.\]
\end{prp}

\begin{proof}
By definition, $\{0\}$ is $\Theta^\leq$-near iff there is some $D\notin\Theta^\leq$ with $D\cup\{0\}\in\Theta^\leq$.  This implies $D=\emptyset$ \textendash\, otherwise we would have $D\cup\{0\}\leq D$ and hence $D\cup\{0\}\in\Theta^\leq$ would imply $D\in\Theta^\leq$, a contradiction.  Thus $\{0\}$ is $\Theta^\leq$-near iff $\emptyset\notin\Theta^\leq$ and $\{0\}\in\Theta^\leq$.  As $D\leq\emptyset$ iff $D=\emptyset$, $\emptyset\notin\Theta^\leq$ iff $\emptyset\notin\Theta$.  Likewise, $D\leq\{0\}$ iff $D=\emptyset$ or $D=\{0\}$, so if $\emptyset\notin\Theta$ then $\{0\}\notin\Theta^\leq$ iff $\{0\}\notin\Theta$.
\end{proof}

Here is another somewhat trivial situation.

\begin{prp}
If $\Theta^\leq$ is weakly admissible then
\[\emptyset\in\Theta\text{ or }\emptyset=\Theta\qquad\Rightarrow\qquad|\mathbb{P}|\leq1.\]
\end{prp}

\begin{proof}
If $\emptyset\in\Theta$ then $\Theta^\leq=\mathcal{P}(\mathbb{P})$ so there are no $\Theta^\leq$-near subsets whatsoever.  If $\emptyset=\Theta$ then $\Theta^\leq=\emptyset$, in which case $\emptyset$ is the only $\Theta^\leq$-near subset.  In either case, $\{p\}$ is not $\Theta^\leq$-near for any $p\in\mathbb{P}$ so weak admissibility then implies that there are no non-zero elements, i.e. $\mathbb{P}$ is either empty or the one-element poset $\{0\}$.
\end{proof}

Under weak admissibility, any $S\subseteq\mathbb{P}$ with a non-zero lower bound is $\Theta^\leq$-near, by \eqref{BoundedNear}.  For finite $S$, we have the following converse, at least when $\Theta$ is directed.

\begin{prp}\label{Near=>Bounded}
If $\Theta$ is directed and $F\subseteq\mathbb{P}$ is finite then
\[F\text{ is $\Theta^\leq$-near}\ \quad\ \Rightarrow\ \quad\ \exists p\in\mathbb{P}\ (F\geq p).\]
Moreover, if $\mathbb{P}$ has a minimum $0$ and $\{0\}$ is not $\Theta^\leq$-near, we can take $p\neq0$ above.
\end{prp}

\begin{proof}
We may assume that $\mathbb{P}$ has a minimum $0$ (by adjoining $0$ to $\mathbb{P}$ if necessary).  If $\emptyset\in\Theta$ then $\Theta^\leq=\mathcal{P}(\mathbb{P})$ so there are no $\Theta^\leq$-near subsets whatsoever, making the result vacuously true.  If $\emptyset\notin\Theta$ then, as $\{0\}$ is not $\Theta^\leq$-near, \autoref{0near} implies that $\{0\}\notin\Theta$.  Now say $\Theta$ is directed and $F$ is $\Theta^\leq$-near, so we have $D\notin\Theta^\leq$ with $D\cup\{f\}\in\Theta^\leq$, for all $f\in F$.  Thus we have $C\in\Theta^\leq$ with $C\leq D\cup\{f\}$, for all $f\in F$.  Note $C\nleq D$, as $D\notin\Theta^\leq$, so we have some $c\in C$ with $c\nleq D$.  As $\{0\}\notin\Theta$, $C\neq\{0\}$ so even if $D=\emptyset$, we can still replace $c$ if necessary to ensure that $c\neq0$.  As $C\leq D\cup\{s\}$, for all $f\in F$, we must have $F\geq c$, as required.
\end{proof}

To state this in a slightly nicer form, let us denote the near singletons by
\[\dt{\mathbb{P}}=\{p\in\mathbb{P}:\{p\}\text{ is $\Theta^\leq$-near}\}.\]
So if we interpret $\mathbb{P}\setminus\{0\}$ just as $\mathbb{P}$ when $\mathbb{P}$ has no minimum then
\[\Theta\text{ is weakly admissible}\qquad\Leftrightarrow\qquad\mathbb{P}\setminus\{0\}\subseteq\dt{\mathbb{P}}.\]

\begin{cor}\label{FiniteNearBounded}
If $\Theta$ is weakly admissible and directed then
\[F\text{ is $\Theta^\leq$-near}\qquad\Leftrightarrow\qquad\exists p\in\dt{\mathbb{P}}\ (F\geq p).\]
\end{cor}

\begin{proof}
The $\Leftarrow$ part is immediate, as is the $\Rightarrow$ part if $\mathbb{P}$ has a minimum $0\in\dt{\mathbb{P}}$.  If $0\notin\dt{\mathbb{P}}$ or if $\mathbb{P}$ has no minimum then the $\Rightarrow$ part follows from \autoref{Near=>Bounded}.
\end{proof}

However, this does not extend to infinite families, as the following example shows.  It also shows that there is no general converse to \autoref{xNear} and that \autoref{NearForAllCovers} does not extend to more general subfamilies of covers.

\begin{xpl}\label{PositiveReals}
Consider the strictly positive real line $X=(0,\infty)$ with its usual metric structure.  Further consider the bases $\mathbb{P}=\{(x,y):0\leq x<y\leq\infty\}$ and $\mathbb{P}_\epsilon=\{(x,y)\in\mathbb{P}:y-x\leq\epsilon\}$ and let $\Theta=\{\mathbb{P}_\epsilon:\epsilon>0\}$ so $\Theta^\leq$ is the family of uniform $\mathbb{P}$-covers of $X$.  Then we see that $D=\{(n-1,n+\frac{1}{n}):n\in\mathbb{N}\}\notin\Theta^\leq$ while $D\cup\{s\}\in\Theta^\leq$, for all $s\in S=\{(x,\infty):0\leq x<\infty\}$.  Thus $S$ is $\Theta^\leq$-near even though $\bigcap S=\emptyset$.  In fact, $S$ is maximal $\Theta^\leq$-near, by \autoref{Near=>Bounded} (because if $(x,y)\in\mathbb{P}\setminus S$ then $y\neq\infty$ so $(y,\infty)\in S$ and $(x,y)\cap(y,\infty)=\emptyset$) even though $S$ is not $\Theta$-Cauchy and hence $S\notin\widehat\Theta$.
\end{xpl}

\section{Generalised Wallman Duality}

Our next goal is to extend Wallman duality to locally compact spaces.

\begin{dfn}
If $X$ has a basis of relatively compact sets, $X$ is \emph{locally compact}.
\end{dfn}

\begin{rmk}
By a relatively compact set $O$, we of course mean that $O$ has compact closure $\mathrm{cl}(O)$.  As a closed subset of a compact set is again compact, it suffices to have a subbasis or even just an open cover of relatively compact sets.  In particular, any compact space is locally compact by this definition.

As is well known, when $X$ is Hausdorff, local compactness is the same as saying each $x\in X$ has a neighbourhood base of compact sets, which is another common definition.  However, even for slightly more general locally Hausdorff spaces, these two definitions are incomparable.  For example, consider an `infinite-point compactification of $\mathbb{N}$', i.e. $X=Y\cup\mathbb{N}$ and
\[\mathcal{O}(X)=\{O\subseteq X:O\subseteq\mathbb{N}\text{ or }O\setminus Y\subseteq_\mathcal{F}\mathbb{N}\}.\]
Then $\{y\}\cup O$, for $O\subseteq_\mathcal{F}\mathbb{N}$, forms a compact neighbourhood base of each $y\in Y$.  However $X$ is not locally compact by the above definition as the closure of any such neighbourhood contains the entirety of $Y$ and is thus not compact.  On the other hand, a compact locally Hausdorff space containing points without neighbourhood bases of compact sets is given in \cite{Scott2013}.
\end{rmk}

For any $S\subseteq\mathbb{P}$, define
\begin{align*}
\mathcal{F}_\Theta^S&=\{F\in\mathcal{F}(\mathbb{P}):F\cup S\text{ is not $\Theta^\leq$-near}\}.\\
\Theta|S&=\{C\subseteq\mathbb{P}:\forall\text{ $\mathcal{F}_\Theta^S$-Cauchy }D\subseteq\mathbb{P}\ (C\cup D\in\Theta^\leq)\}.
\end{align*}
This `restriction' of $\Theta$ to $S$ represents a family of covers of the closure of $\widehat\Theta_S$.

\begin{thm}\label{pClosure}
For any $S\subseteq\mathbb{P}$,
\[\mathrm{cl}(\widehat{\Theta}_S)\subseteq\widehat{\Theta|S}\subseteq\widehat\Theta.\]
\end{thm}

\begin{proof}
For the first $\subseteq$, take $R\in\mathrm{cl}(\widehat{\Theta}_S)$ so, for every finite $F\subseteq R$,
\[\widehat\Theta_F\cap\widehat\Theta_S\neq\emptyset.\]
Thus we have $Q\in\widehat\Theta$ with $F\cup S\subseteq Q$ and hence $F\cup S$ is $\Theta^\leq$-near, by \autoref{xNear}.  In other words, for any finite $F\subseteq\mathbb{P}$ such that $F\cup S$ is not $\Theta^\leq$-near, $F\setminus R\neq\emptyset$, i.e. $\mathbb{P}\setminus R$ is $\mathcal{F}^S_\Theta$-Cauchy.  Thus, for any $C\in\Theta|S$, $C\cup(\mathbb{P}\setminus R)\in\Theta^\leq$ and hence, as $R$ is $\Theta^\leq$-Cauchy, $C\cap R\neq\emptyset$, i.e. $R$ is $\Theta|S$-Cauchy.  Also, as $R$ is $\Theta$-round and $\Theta\subseteq\Theta|S$, it follows immediately that $R$ is $\Theta|S$-round and hence $R\in\widehat{\Theta|S}$.

For the next $\subseteq$, take $R\in\widehat{\Theta|S}$.  We claim that $\mathbb{P}\setminus R$ is $\mathcal{F}^S_\Theta$-Cauchy, which is saying that $F\cup S$ is $\Theta^\leq$-near, for all finite $F\subseteq R$.  To see this note that, as $R$ is $\Theta|S$-round, $(\mathbb{P}\setminus R)\cup\{f\}\in\Theta|S$, for all $f\in F\subseteq R$.  If $F\cup S$ were not $\Theta^\leq$-near then every $\mathcal{F}^S_\Theta$-Cauchy $D$ would contain some $f\in F$ and hence we would have $(\mathbb{P}\setminus R)\cup D=(\mathbb{P}\setminus R)\cup\{f\}\cup D\in\Theta^\leq$, as $(\mathbb{P}\setminus R)\cup\{f\}\in\Theta|S$.  But this would mean $(\mathbb{P}\setminus R)\in\Theta|S$, contradicting the fact that $R$ is $\Theta|S$-Cauchy.  This proves the claim.

As $R$ is $\Theta|S$-round, $\mathbb{P}\setminus R\cup\{r\}\in\Theta|S$, for every $r\in R$, and hence, by the claim,
\[\mathbb{P}\setminus R\cup\{r\}=\mathbb{P}\setminus R\cup\{r\}\cup\mathbb{P}\setminus R\in\Theta^\leq.\]
Thus $R$ is $\Theta$-round.  Also, as $R$ is $\Theta|S$-Cauchy and $\Theta\subseteq\Theta|S$, it follows immediately that $R$ is $\Theta$-Cauchy and hence $R\in\widehat\Theta$.
\end{proof}

In particular, \autoref{pClosure} yields $\widehat\Theta=\widehat{\Theta|\emptyset}$, i.e. we can always enlarge $\Theta$ to $\Theta|\emptyset$ without affecting the spectrum.  The advantage of doing this is that we can piece together covers in each $\Theta|p$ to obtain covers in $\Theta|\emptyset$, as the following result shows.

\begin{prp}\label{CoverPatching}
As long as $\Theta\neq\emptyset$,
\[\Theta|\emptyset=\bigcap_{p\in\mathbb{P}}\Theta|p.\]
\end{prp}

\begin{proof}
First note that, for any $Q,S\subseteq\mathbb{P}$,
\[Q\subseteq S\qquad\Rightarrow\qquad\mathcal{F}^Q_\Theta\subseteq\mathcal{F}^S_\Theta\qquad\Rightarrow\qquad\Theta|Q\subseteq\Theta|S.\]
In particular, $\Theta|\emptyset\subseteq\bigcap_{p\in\mathbb{P}}\Theta|p$.  For the reverse inclusion, take $C\in\bigcap_{p\in\mathbb{P}}\Theta|p$ and $\mathcal{F}^\emptyset_\Theta$-Cauchy $D$.  If $D=\mathbb{P}$ then $C\cup D\in\Theta^\leq$, as $\Theta\neq\emptyset$.  If $D\neq\mathbb{P}$ then we have $p\in\mathbb{P}\setminus D$ and $D$ is immediately seen to be $\mathcal{F}^p_\Theta$-Cauchy which again yields $C\cup D\in\Theta^\leq$, as $C\in\Theta|p$.  As $D$ was arbitrary, this shows that $C\in\Theta|\emptyset$.
\end{proof}

The $\Theta\neq\emptyset$ condition here really is necessary.  Indeed, if $\Theta=\emptyset$ then $\emptyset$ is the only $\Theta^\leq$-near subset so $\mathcal{F}^\emptyset_\Theta=\mathcal{P}(\mathbb{P})\setminus\{\emptyset\}$ while $\mathcal{F}^p_\Theta=\mathcal{P}(\mathbb{P})$, for all $p\in\mathbb{P}$.  This means that $\mathbb{P}$ is $\mathcal{F}^\emptyset_\Theta$-Cauchy even though $\mathbb{P}\notin\Theta^\leq$ and hence $\Theta|\emptyset=\emptyset$.  On the other hand, for all $p\in\mathbb{P}$, there are no $\mathcal{F}^p_\Theta$-Cauchy subsets because $\emptyset\in\mathcal{F}^p_\Theta$, which means the defining property of $\Theta|p$ holds vacuously for arbitrary subsets, i.e. $\Theta|p=\mathcal{P}(\mathbb{P})$.

For general $S$, \autoref{pClosure} only yields the inclusion $\widehat{\Theta|S}\subseteq\mathrm{cl}(\widehat\Theta_S)$.  For finite $S$, we can strengthen this to an equality under an appropriate additional assumption.

\begin{dfn}
We call $\Theta$ \emph{non-degenerate} if, for all finite $F\subseteq\mathbb{P}$,
\[\tag{Non-Degenerate}F\text{ is $\Theta^\leq$-near}\qquad\Rightarrow\qquad\widehat\Theta_F\neq\emptyset.\]
\end{dfn}

\begin{prp}\label{Non-DegenerateResults}
If $\Theta$ is non-degenerate, for all $F\in\mathcal{F}(\mathbb{P})$ and $\mathcal{F}^F_\Theta$-Cauchy $D$,
\[\mathrm{cl}(\widehat\Theta_F)=\widehat{\Theta|F}\qquad\text{and}\qquad\widehat\Theta\setminus\mathrm{cl}(\widehat\Theta_F)\subseteq\bigcup_{d\in D}\widehat\Theta_d.\]
\end{prp}

\begin{proof}
Say we had $R\in\widehat{\Theta|F}\setminus\mathrm{cl}(\widehat\Theta_F)$.  In particular, $\widehat\Theta\setminus\mathrm{cl}(\widehat\Theta_F)$ is a neighbourhood of $R$ so we have finite $G\subseteq R$ with $\widehat\Theta_F\cap\widehat\Theta_G=\emptyset$.  As $\Theta$ is non-degenerate, this means $F\cup G$ is not $\Theta^\leq$-near, i.e. $G\in\mathcal{F}^F_\Theta$.  Thus $D$ must contain some $g\in G$ and hence $(\mathbb{P}\setminus R)\cup D=(\mathbb{P}\setminus R)\cup\{g\}\cup D\in\Theta|F$, as $R$ is $\Theta|F$-round and $g\in G\subseteq R$.  As $D$ was arbitrary, this means $\mathbb{P}\setminus R\in\Theta|F$, contradicting the fact that $R$ is also $\Theta|F$-Cauchy.  Combined with \autoref{pClosure}, this proves the first equality.

Likewise, for the second inclusion, take any $R\in\widehat\Theta\setminus\mathrm{cl}(\widehat\Theta_F)$, so we have finite $G\subseteq R$ with $\widehat\Theta_F\cap\widehat\Theta_G=\emptyset$.  As $\Theta$ is non-degenerate, this implies that $F\cup G$ is not $\Theta^\leq$-near, i.e. $G\in\mathcal{F}^F_\Theta$ and hence $G\cap D\neq\emptyset$ so $R\in\bigcup_{d\in D}\widehat\Theta_d$.
\end{proof}

If $\Theta$ consists of arbitrary covers then so do its restrictions.

\begin{prp}\label{Theta|pForAllCovers}
If $X$ is a $T_1$ space, $\leq\ =\ \subseteq$ on a subbasis $\mathbb{P}\subseteq\mathcal{O}(X)$ and $\Theta$ is a coinitial subset of $\mathcal{C}_X(\mathbb{P})$ then, for all $S\subseteq\mathbb{P}$,
\[\Theta|S=\mathcal{C}_{\mathrm{cl}(\bigcap S)}(\mathbb{P}).\]
\end{prp}

\begin{proof}
By \autoref{T1Recovery} we can identify $X$ and $\widehat\Theta$.  It then follows from \autoref{pClosure} that $\Theta|S\subseteq\mathcal{C}_{\mathrm{cl}(\bigcap S)}(\mathbb{P})$.  By \autoref{NearForAllCovers}, $\Theta$ is non-degenerate so it follows from \autoref{Non-DegenerateResults} that $\mathrm{cl}(\bigcap S)\cup\bigcup D=X$, for any $\mathcal{F}^S_\Theta$-Cauchy $D\subseteq\mathbb{P}$, at least if $S$ is finite.  But as \autoref{NearForAllCovers} applies to infinite sets, arguing as in the proof of \autoref{Non-DegenerateResults} yields the same result for infinite $S$.  So $\bigcup C\cup\bigcup D=X$, for any $C\in\mathcal{C}_{\mathrm{cl}(\bigcap S)}(\mathbb{P})$ and $\mathcal{F}^S_\Theta$-Cauchy $D\subseteq\mathbb{P}$, and hence $C\in\Theta|S$.  As $C$ was arbitrary, this shows that $\mathcal{C}_{\mathrm{cl}(\bigcap S)}(\mathbb{P})\subseteq\Theta|S$.
\end{proof}

\begin{thm}\label{ThetaLocallyCompactSubbases}
If $\Theta|p\cap\mathcal{F}(\mathbb{P})$ is coinitial in $\Theta|p$, for all $p\in\mathbb{P}$, then $\widehat\Theta$ is locally compact.  As long as $\Theta\neq\emptyset$ too then, for all $p,q\in\mathbb{P}$ and $C\subseteq\mathbb{P}$,
\begin{align}
\label{ThetaPreorder=subseteqGeneral}\widehat\Theta_p\subseteq\widehat\Theta_q\qquad&\Leftrightarrow\qquad p\leq_{\Theta|\emptyset}q.\\
\label{CoinitialCoversGeneral}\widehat\Theta\subseteq\bigcup_{c\in C}\widehat{\Theta}_c\qquad&\Leftrightarrow\qquad C\in\Theta|\emptyset.
\end{align}
\end{thm}

\begin{proof}
For all $p\in\mathbb{P}$, $\widehat{\Theta|p}$ is compact, by \autoref{ThetaCompact}, and hence $\mathrm{cl}(\widehat\Theta_p)$ is compact, by \autoref{pClosure}.  Thus $\widehat\Theta$ is locally compact.
\begin{itemize}
\item[\eqref{CoinitialCoversGeneral}] Note that $\widehat\Theta\subseteq\bigcup_{c\in C}\widehat{\Theta}_c$ implies $\widehat{\Theta|p}\subseteq\bigcup_{c\in C}\widehat{\Theta}_c$, for all $p\in\mathbb{P}$.  By \eqref{CoinitialCovers}, that means $C\in\bigcap_{p\in\mathbb{P}}\Theta|p=\Theta|\emptyset$, by \autoref{CoverPatching}.

\item[\eqref{ThetaPreorder=subseteqGeneral}] This follows immediately from \eqref{CoinitialCoversGeneral}, as in the proof of \eqref{ThetaPreorder=subseteq}.\qedhere
\end{itemize}
\end{proof}

We can now state a natural generalisation of Wallman duality.  First let us call $(\mathbb{P},\leq,\Theta)$ a \emph{generalised Wallman poset} if, for all $p\in\mathbb{P}$,
\begin{enumerate}
\item\label{RestrictionFinite} $\Theta|p\cap\mathcal{F}(\mathbb{P})$ is coinitial in $\Theta|p$,
\item\label{StronglyWallmanAdmissible} $\Theta=\Theta|\emptyset$ and $\leq\ =\ \leq_\Theta$.
\end{enumerate}
Note here that \eqref{RestrictionFinite} is a kind of local finiteness condition, while \eqref{StronglyWallmanAdmissible} is a strengthening of Wallman admissibility, by \autoref{WallmanAdmissibleEquivalent}.  In summary, we have the following.

\begin{thm}\label{GeneralisedWallmanDualitySummary}\
\begin{center}
\textbf{\upshape Generalised Wallman posets are dual to\\ relatively compact subbases of locally compact $T_1$ spaces}.
\end{center}
\end{thm}
More precisely, any generalised Wallman poset $(\mathbb{P},\leq,\Theta)$ can be concretely represented as the subbasis $(\widehat\Theta_p)_{p\in\mathbb{P}}$ of the locally compact $T_1$ space $\widehat\Theta$ where $\leq$ becomes $\subseteq$ and $\Theta$ becomes the family of all covers $(\{\widehat\Theta_c:c\in C\})_{C\in\Theta}$ of $\widehat\Theta$, by \autoref{T1} and \autoref{ThetaLocallyCompactSubbases}.  On the other hand, if $\mathbb{P}$ is a subbasis of relatively compact subsets of a $T_1$ space $X$, $\leq\ =\ \subseteq$ on $\mathbb{P}$ and $\Theta=\mathcal{C}_X(\mathbb{P})$ then $(\mathbb{P},\leq,\Theta)$ is a generalised Wallman poset, by \autoref{SubbasisOrder} and \autoref{Theta|pForAllCovers}.  Moreover, in this case $\widehat\Theta$ is homeomorphic to the original space $X$, by \autoref{T1Recovery}.

The drawback of this duality is that $\Theta=\Theta|\emptyset$ rules out covering families commonly considered on locally compact spaces, such as the uniform open covers on a metric space like $\mathbb{R}$.  To include such examples, it would be better if we could weaken $\Theta=\Theta|\emptyset$ to something like $\Theta=\Theta^\leq$.  The problem is that it is no longer clear then that $\Theta$ will be faithful or non-degenerate.  However, as long as we are willing to work with bases rather than subbases, we can rectify this by placing another assumption on $\Theta$, one that still applies to the uniform covers example.  Specifically, we can assume that $\Theta$ is also directed, as we show in the next section.

\begin{rmk}
Even when $\Theta$ is not directed, there is a natural directed replacement.  This allows us to translate results about directed $\Theta$ into more general results, at least in theory.  Here is an outline of how this would be done.

First we replace $\mathbb{P}$ with $\mathcal{F}(\mathbb{P})$ ordered by
\[F\leq^\mathcal{F}G\qquad\Leftrightarrow\qquad\bigcap_{f\in F}\Theta^{\leq f}\subseteq\bigcap_{g\in G}\Theta^{\leq g}.\]
Each $F\in\mathcal{F}(\mathbb{P})$ is meant to represent the intersection of those $f\in F$.  Indeed, in this ordering, each $F\in\mathcal{F}(\mathbb{P})$ is a meet of its singleton subsets, i.e. $\mathcal{F}(\mathbb{P})$ would be a meet-semilattice except that reflexivity may not hold so, strictly speaking, $\mathcal{F}(\mathbb{P})$ is only a proset, not a poset.  However all the same definitions apply and, if desired, one can consider the quotient poset of equivalent classes.  Next, we replace $\Theta$ with
\[\Theta^\mathcal{F}=\{\Phi\subseteq\mathcal{F}(\mathbb{P}):\forall\text{ $\Phi$-Cauchy }C\subseteq\mathbb{P}\ (C\in\Theta)\}.\]
Reflexivity aside, $\Theta^\mathcal{F}$ is also a meet-semilattice and hence $\leq^\mathcal{F}$-directed.  Specifically,
\[\Phi\wedge\Psi=\{F\cup G:F\in\Phi\text{ and }G\in\Psi\}.\]
Then one can check that the spectrum remains unchanged, more precisely
\[\widehat{\Theta^\mathcal{F}}=\{\mathcal{F}(R):R\in\widehat\Theta\}.\]
The main thing to be aware of though is that, just like $\Theta|\emptyset$,
\[\Theta^{\mathcal{F}\leq^\mathcal{F}}\cap\mathcal{P}(\mathbb{P})\]
(identifying each $S\subseteq\mathbb{P}$ with $\{\{s\}:s\in S\}$) could be strictly bigger than $\Theta^\leq$.
\end{rmk}

\section{Basic Wallman Duality}\label{BasicWallmanDuality}

First we note that bases are closely related to directed $\Theta$.

\begin{prp}\label{Directed=>Basis}
If $\Theta$ is directed then $(\widehat\Theta_p)_{p\in\mathbb{P}}$ is a basis for $\widehat\Theta$.
\end{prp}

\begin{proof}
As $\Theta$ is directed, every $R\in\widehat\Theta$ will be a filter, by \autoref{Directed=>Directed}.  So if $R\in\widehat\Theta_p\cap\widehat\Theta_q$, i.e. $p,q\in R$, then we have $r\in R$ with $p,q\geq r$ and hence $R\in\widehat\Theta_r\subseteq\widehat\Theta_p\cap\widehat\Theta_q$, i.e. $(\widehat\Theta_p)_{p\in\mathbb{P}}$ is a basis.
\end{proof}

\begin{prp}
If $\mathbb{P}\subseteq\mathcal{O}(X)$ is a basis and $\leq\ =\ \subseteq$ on $\mathbb{P}$,  $\mathcal{C}_X(\mathbb{P})$ is directed.
\end{prp}

\begin{proof}
If $C,D\in\mathcal{C}_X(\mathbb{P})$ then, for each $x\in X$, we have $c\in C$ and $d\in D$ with $x\in c\cap d$.  If $\mathbb{P}$ is a basis then we have $p_x\in\mathbb{P}$ with $x\in p_x\subseteq c\cap d$.  Thus $(p_x)_{x\in X}$ is a cover of $X$ refining both $C$ and $D$, i.e. $\mathcal{C}_X(\mathbb{P})$ is directed.
\end{proof}

Near subsets again play an important role, near pairs in particular.

\begin{dfn}
For any $C\subseteq\mathbb{P}$ and $p\in\mathbb{P}$, the \emph{$C$-star of $p$} is given by
\[\tag{$C$-star}Cp=\{c\in C:\{c,p\}\text{ is $\Theta^\leq$-near}\}.\]
\end{dfn}

This notation comes from \cite[VIII.2.1]{PicadoPultr2012}, although there the $C$-star of $p$ is instead defined to be $\bigvee Cp$.  This is fine for frames, but our posets may not have any lattice structure so we must instead work with subsets rather than joins.

For any $p\in\mathbb{P}$, let us define
\[\Theta p=\{Cp:C\in\Theta\}.\]
This provides an alternative `restriction' of $\Theta$ to $p$.

\begin{prp}\label{Restrictions}
For any $p\in\mathbb{P}$, $\Theta p\subseteq\Theta|p$.
\end{prp}

\begin{proof}
For any $C\in\Theta$, by definition $C\setminus Cp$ consists of all those $c\in C$ such that $\{c,p\}$ is not $\Theta^\leq$-near.  Thus $\{\{c\}:c\in C\setminus Cp\}\subseteq\mathcal{F}^p_\Theta$ and hence any $\mathcal{F}^p_\Theta$-Cauchy $D\subseteq\mathbb{P}$ must contain $C\setminus Cp$.  Thus $Cp\cup D\in\Theta^\leq$, for all such $D$, i.e. $Cp\in\Theta|p$.
\end{proof}

It follows that $(\Theta p)^\leq\subseteq\Theta|p$, but this inclusion could be strict.  For example, if $\Theta^\leq$ is weakly admissible and $\mathbb{P}$ has a maximum $1$ but no minimum then $\mathbb{P}1=\mathbb{P}$ and hence $\Theta1=\Theta$.  However $\Theta|1=\Theta|\emptyset$ will often be bigger than $\Theta^\leq$, e.g. if $\mathbb{P}$ consists of all the non-empty open subsets of $\mathbb{R}$ and $\Theta$ consists of all the uniform covers of $\mathbb{R}$ then $\Theta|1=\Theta|\emptyset$ will include all arbitrary covers of $\mathbb{R}$.

Again taking notation from \cite[VIII.1.1]{PicadoPultr2012}, for any $C\subseteq\mathcal{P}(X)$ and $p\subseteq X$ let
\begin{align*}
C*p&=\{c\in C:c\cap p\neq\emptyset\}.\\
\Theta*p&=\{C*p:C\in\Theta\}.
\end{align*}
Let us call $\Theta\subseteq\mathcal{C}_X(\mathbb{P})$ \emph{$*$-coinitial} if $\mathcal{C}_X(\mathbb{P})\subseteq(\Theta*p)^\leq$, for all $p\in\mathbb{P}$.

\begin{thm}\label{LocalT1Recovery}
If $X\neq\emptyset$ is a $T_1$ space, $\leq\ =\ \subseteq$ on a subbasis $\mathbb{P}\subseteq\mathcal{O}(X)$ and $\Theta\subseteq\mathcal{C}_X(\mathbb{P})$ is directed and $*$-coinitial then $X$ is homeomorphic to $\widehat\Theta$ via
\[x\mapsto\mathbb{P}_x=\{N\in\mathbb{P}:x\in N\}.\]
It follows that $\mathbb{P}$ is a basis and, for all $p\in\mathbb{P}$,
\[(\Theta*p)^\leq=(\Theta p)^\leq=\mathcal{C}_{\mathrm{cl}(p)}(\mathbb{P}).\]
\end{thm}

\begin{proof}
As in the proof of \autoref{T1Recovery}, we immediately see that $x\mapsto\mathbb{P}_x$ is an injective homeomorphism onto a subspace of $\mathcal{P}(\mathbb{P})$, each $\mathbb{P}_x$ is $\Theta$-Cauchy and, for any $x\in X$, $\mathbb{P}\setminus\mathbb{P}_x$ covers $X\setminus\{x\}$.  So, for each $p\in\mathbb{P}_x$, $\{p\}\cup(\mathbb{P}\setminus\mathbb{P}_x)\in\mathcal{C}_X(\mathbb{P})$.  Thus we have $C\in\Theta$ with $C*p\leq\{p\}\cup(\mathbb{P}\setminus\mathbb{P}_x)$ and hence
\[\mathbb{P}_x\cap C=\mathbb{P}_x\cap C*p\leq\mathbb{P}_x\cap(\{p\}\cup(\mathbb{P}\setminus\mathbb{P}_x))=\{p\}.\]
This shows that $\mathbb{P}_x$ is $\Theta$-round and hence $\mathbb{P}_x\in\widehat\Theta$.

On the other hand, any $R\in\widehat\Theta$ must contain $\mathbb{P}_x$, for some $x\in X$.  To see this, first note that, by \eqref{Empty=Theta}, $R\neq\emptyset$, as $\Theta\neq\emptyset$, so we can take $r\in R$.  If $\emptyset\in R$ then $R=\emptyset^\leq=\mathbb{P}\supseteq\mathbb{P}_x$, for any $x\in X$, and we are done.  So we may assume $\emptyset\notin R$ and hence $R\subseteq\mathbb{P}*r$, as $R$ is directed by \autoref{Directed=>Directed}.  If $R$ did not contain $\mathbb{P}_x$, for any $x\in X$, then we would have $\mathbb{P}\setminus R\in\mathcal{C}_X(\mathbb{P})$ and $*$-coinitiality would yield $C\in\Theta$ with $C*r\leq\mathbb{P}\setminus R$.  As $R\subseteq\mathbb{P}*r$ is $\Theta$-Cauchy, we $c\in C\cap R\subseteq C*r\leq\mathbb{P}\setminus R$, contradicting the fact that $R$ is an up-set.  This shows that $\mathbb{P}_x\subseteq R$, for some $x\in X$, and hence $R=\mathbb{P}_x$, as we already showed that $\mathbb{P}_x\in\widehat\Theta$, i.e.
\[\widehat\Theta=\{\mathbb{P}_x:x\in X\}.\]

To prove $\Theta*p=\Theta p$, it suffices to show that
\[\{p,q\}\text{ is $\Theta^\leq$-near}\qquad\Leftrightarrow\qquad p\cap q\neq\emptyset.\]
To see this, note that if $p\cap q\neq\emptyset$ then $p,q\in\mathbb{P}_x\in\widehat\Theta$ so $\{p,q\}$ is $\Theta^\leq$-near, by \autoref{xNear}.  For the converse, we argue as in the proof of \autoref{Near=>Bounded}.  Specifically, if $\{p,q\}$ is $\Theta^\leq$-near then we have $D\notin\Theta^\leq$ with $D\cup\{p\},D\cup\{q\}\in\Theta^\leq$.  As $\Theta$ is directed, we have $C\in\Theta$ with $C\leq D\cup\{p\}$ and $C\leq D\cup\{q\}$.  As $D\notin\Theta^\leq$, we know that $C\nleq D$, i.e. we have $c\in C$ with $c\nleq D$.  As $X\neq\emptyset$, $C$ must contain a non-empty set so we may assume that $c\neq\emptyset$.  As $C\leq D\cup\{p\}$ and $C\leq D\cup\{q\}$, we have $p,q\geq c$ and hence $p\cap q\supseteq c\neq\emptyset$, as required.

By \autoref{Directed=>Basis}, $\mathbb{P}$ is a basis (although this can also be checked directly).  By the definition of topological closure, if $p,q\in\mathbb{P}$ and $x\in\mathrm{cl}(p)\cap q$ then $p\cap q\neq\emptyset$, from which it follows that $(\Theta*p)^\leq\subseteq\mathcal{C}_{\mathrm{cl}(p)}(\mathbb{P})$.  Conversely, if $C\in\mathcal{C}_{\mathrm{cl}(p)}(\mathbb{P})$ then, as $\mathbb{P}$ is a basis, $C\cup(\mathbb{P}\setminus\mathbb{P}*p)\in\mathcal{C}_X(\mathbb{P})$.  By $*$-coinitiality, we have $D\in\Theta$ with $D*p\leq C\cup(\mathbb{P}\setminus\mathbb{P}*p)$ and hence $D*p\leq C$, i.e.
\[(\Theta*p)^\leq=\mathcal{C}_{\mathrm{cl}(p)}(\mathbb{P}).\qedhere\]
\end{proof}

Next we have an analog of \autoref{pClosure}

\begin{prp}\label{pClosure2}
For any $p\subseteq\mathbb{P}$,
\[\mathrm{cl}(\widehat{\Theta}_p)\subseteq\widehat{\Theta p}\subseteq\widehat\Theta.\]
If $\Theta$ is directed and non-degenerate then $\mathrm{cl}(\widehat{\Theta}_p)=\widehat{\Theta p}$.
\end{prp}

\begin{proof}
For the first $\subseteq$, take $R\in\mathrm{cl}(\widehat{\Theta}_p)$ so, for every $r\in R$,
\[\widehat\Theta_r\cap\widehat\Theta_p\neq\emptyset.\]
Thus we have $Q\in\widehat\Theta$ with $p,r\in Q$ and hence $\{p,r\}$ is $\Theta^\leq$-near, by \autoref{xNear}.  This shows that $R\subseteq\mathbb{P}p$ and hence $R$ is $\Theta p$-Cauchy, as $R$ is $\Theta$-Cauchy.  Also, as $R$ is $\Theta$-round, it follows immediately that $R$ is $\Theta p$-round and hence $R\in\widehat{\Theta p}$.

For the next $\subseteq$, take $R\in\widehat{\Theta p}$.  As $R$ is $\Theta p$-Cauchy, $R$ is certainly $\Theta$-Cauchy.  As $R$ is also $\Theta p$-round, it follows that $R\subseteq\mathbb{P}p$ and hence $R$ is also $\Theta$-round, i.e. $R\in\widehat\Theta$.  Also $R\subseteq\mathbb{P}p$ means that, for all $r\in R$, $\{r,p\}$ is $\Theta^\leq$-near and hence $\widehat\Theta_r\cap\widehat\Theta_p\neq\emptyset$, as $\Theta$ is non-degenerate.  If $\Theta\neq\emptyset$ is also directed then $(\widehat{\Theta}_p)_{p\in\mathbb{P}}$ is a basis, by \autoref{Directed=>Basis}, so this shows that $R\in\mathrm{cl}(\widehat{\Theta}_p)$, i.e. $\mathrm{cl}(\widehat{\Theta}_p)=\widehat{\Theta p}$.
\end{proof}

Note that for directed $\Theta$ to non-degenerate, it suffices that $\widehat\Theta_p\neq\emptyset$, for all $p\in\mathbb{P}$ (except $0$, if $\mathbb{P}$ has a minimum $0$ and $\{0\}$ is not $\Theta^\leq$-near), by \autoref{Near=>Bounded}.

Let us call $\Theta$ \emph{nearly finite} if $\Theta p\cap\mathcal{F}(\mathbb{P})$ is coinitial in $\Theta p$, for all $p\in\mathbb{P}$.

\begin{thm}\label{ThetaLocallyCompact}
If $\Theta$ is directed, weakly admissible and nearly finite then $\widehat\Theta$ is locally compact, $\Theta$ is non-degenerate and, for all $p,q\in\mathbb{P}$,
\begin{equation}
\label{ThetaPreorder=subseteq'}\widehat\Theta_p\subseteq\widehat\Theta_q\qquad\Leftrightarrow\qquad p\leq_{\Theta^\leq}q.
\end{equation}
Moreover, $(\{\widehat\Theta_c:c\in C\})_{C\in\Theta}$ is $*$-coinitial in $\mathcal{C}_{\widehat\Theta}(\{\widehat\Theta_p:p\in\mathbb{P}\})$.
\end{thm}

\begin{proof}
As $\Theta$ is nearly finite, for all $p\in\mathbb{P}$, the spectrum of $\Theta p$ is the same as the spectrum of $\Theta p\cap\mathcal{F}(\mathbb{P})$, which is compact, by \autoref{ThetaCompact}.  Thus $\mathrm{cl}(\widehat{\Theta}_p)\subseteq\widehat{\Theta p}$ is also compact, by \autoref{pClosure2}.  By \autoref{Directed=>Basis}, $(\widehat\Theta_p)_{p\in\mathbb{P}}$ is thus a basis of relatively compact sets so $\widehat\Theta$ is locally compact.

The $\Leftarrow$ part of \eqref{ThetaPreorder=subseteq'} appears in \autoref{FaithfulRightarrow}.  Conversely, say $p\nleq_{\Theta^\leq}q$, so we have some $S\subseteq\mathbb{P}$ with $S\cup\{p\}\in\Theta^\leq$ but $S\cup\{q\}\notin\Theta^\leq$.  We claim that $S\cup\{q\}\notin(\Theta p)^\leq$, otherwise we would have $S\cup\{q\}\cup(\mathbb{P}\setminus\mathbb{P}p)\in\Theta^\leq$.  As $\Theta$ is directed, we would then have $C\in\Theta$ with $C\leq S\cup\{p\}$ and $C\leq S\cup\{q\}\cup(\mathbb{P}\setminus\mathbb{P}p)$.  Note $C\nleq S\cup\{q\}$, as $S\cup\{q\}\notin\Theta^\leq$, so we have $c\in C$ with $c\nleq S\cup\{q\}$ and hence $c^\leq\neq\mathbb{P}$ and $c\leq p$, as $c\leq S\cup\{p\}$.  By weak admissibility, $\{c\}$ and hence $\{c,p\}$ is $\Theta^\leq$-near so $c\nleq\mathbb{P}\setminus\mathbb{P}p$, contradicting $C\leq S\cup\{q\}\cup(\mathbb{P}\setminus\mathbb{P}p)$, thus proving the claim.  By \eqref{CoinitialCovers}, we have $R\in\widehat{\Theta p}\subseteq\widehat\Theta$ disjoint from $S\cup\{q\}$ and hence $p\in R$, as $S\cup\{p\}\in\Theta$, i.e. $R\in\widehat\Theta_p\setminus\widehat\Theta_q$.  This proves the $\Rightarrow$ part of \eqref{ThetaPreorder=subseteq'}.

If $\Theta$ is weakly admissible but $\{p\}$ is not $\Theta^\leq$-near, for any $p\in\mathbb{P}$, then $|\mathbb{P}|\leq1$ and non-degeneracy amounts to showing $\widehat\Theta\neq\emptyset$ if $\emptyset\notin\Theta$, which again follows from \eqref{CoinitialCovers}.  Otherwise, to verify non-degeneracy it suffices to consider non-empty finite $\Theta^\leq$-near $F\subseteq\mathbb{P}$.  So we have $S\notin\Theta^\leq$ such that, for all $f\in F$, $S\cup\{f\}\in\Theta^\leq$.  Fixing $f\in F$, we again see that $S\notin(\Theta f)^\leq$ so \eqref{CoinitialCovers} yields $R\in\widehat{\Theta f}\subseteq\widehat\Theta$ disjoint from $S$ and hence with $F\subseteq R$, as required.

By non-degeneracy and \autoref{xNear},
\[\{p,q\}\text{ is $\Theta^\leq$-near}\qquad\Leftrightarrow\qquad\widehat\Theta_p\cap\widehat\Theta_q\neq\emptyset.\]
It follows that $(\{\widehat\Theta_c:c\in C\})_{C\in\Theta}$ is $*$-coinitial in $\mathcal{C}_{\widehat\Theta}(\{\widehat\Theta_p:p\in\mathbb{P}\})$, as \eqref{CoinitialCovers} shows that $(\{\widehat\Theta_c:c\in C\})_{C\in\Theta q}$ is coinitial in $\mathcal{C}_{\widehat{\Theta q}}(\{\widehat{\Theta q}_p:p\in\mathbb{P}\})$.
\end{proof}

To obtain a precise duality, we need the following extra observation.

\begin{prp}\label{ThetapNonempty}
If $\Theta$ is a weakly admissible filter and $p^\leq\neq\mathbb{P}$ then $\emptyset\notin\Theta p$.
\end{prp}

\begin{proof}
Take any $C\in\Theta$ and non-zero $p\in\mathbb{P}$.  By weak admissibility, we have $D\notin\Theta$ with $D\cup\{p\}\in\Theta$.  As $\Theta$ is directed, we have $E\in\Theta$ with $E\leq C$ and $E\leq D\cup\{p\}$.  As $D\notin\Theta=\Theta^\leq$, we have $e\in E$ with $e\nleq D$ and hence $e\leq p$.  If we had $E=\{0\}$ and hence $D=\emptyset$ then all subsets would be $\Theta$-near so $e\leq C=Cp$ and, in particular, $Cp\neq\emptyset$.  If $E\neq\{0\}$ then we can take $e$ to be non-zero and hence $\{e\}$ is $\Theta$-near, by weak admissibility, so $e\leq Cp$ and again $Cp\neq\emptyset$.
\end{proof}

We can now state an extension of Wallman duality for bases as follows.

\begin{thm}\label{BasicWallmanDualitySummary}
\
\begin{center}
\textbf{\upshape Abstract nearly finite proper Wallman admissible filters $\Theta\subseteq\mathcal{P}(\mathbb{P})$ are dual to concrete $*$-coinitial filters $\Theta\subseteq\mathcal{C}_X(\mathbb{P})$ where $\mathbb{P}\subseteq\mathcal{O}(X)$ is a basis of relatively compact sets on a locally compact $T_1$ space $X\neq\emptyset$}.
\end{center}
\end{thm}

More precisely, if $\Theta$ is a nearly finite Wallman admissible filter then $\Theta$ can be concretely represented as the basis $(\widehat\Theta_p)_{p\in\mathbb{P}}$ of relatively compact sets on the locally compact $T_1$ space $\widehat\Theta$ where $\leq$ becomes $\subseteq$ and $\Theta$ becomes the $*$-coinitial family of covers $(\{\widehat\Theta_c:c\in C\})_{C\in\Theta}$ of $\widehat\Theta$, by \autoref{T1}, \autoref{WallmanAdmissibleEquivalent}, \autoref{Wallman=>Weakly} and \autoref{ThetaLocallyCompact}.  Moreover, if $|\mathbb{P}|\geq2$ then we have non-zero $p\in\mathbb{P}$ and $\emptyset\notin\Theta p$, by \autoref{ThetapNonempty}, and hence $\emptyset\neq\widehat{\Theta p}\subseteq\widehat\Theta$, again by \eqref{ThetaCompact} and \autoref{pClosure2}.  As $\Theta$ is proper, i.e. $\Theta\neq\mathcal{P}(\mathbb{P})$, the only other possibility is $\mathbb{P}=\{0\}$ and $\Theta=\{\{0\}\}=\widehat\Theta$, again showing that $\widehat\Theta\neq\emptyset$.  On the other hand, if $\mathbb{P}$ is a basis of relatively compact sets on a $T_1$ space $X\neq\emptyset$ and $\Theta\subseteq\mathcal{C}_X(\mathbb{P})$ is a $*$-coinitial filter then $\Theta$ is nearly finite and Wallman admissible and $\widehat\Theta$ is homeomorphic to the original space $X$, by \autoref{LocalT1Recovery}.

\section{Order Covers}\label{OrderCovers}

We wish to compare Wallman admissibility with a slightly stronger notion due to Picado and Pultr.  To define a version of this that applies to general posets, not just frames, we first examine an order theoretic analog of a cover.

\begin{dfn}
We call $C\subseteq\mathbb{P}$ a \emph{$\leq$-cover} if, for all $p,q\in\mathbb{P}$,
\[\tag{$\leq$-Cover}\label{leqCover}p\leq q\qquad\Leftrightarrow\qquad p^\geq\cap C^\geq\subseteq q^\geq.\]
We denote the collection of all $\leq$-covers by
\[\mathcal{C}_\leq(\mathbb{P})=\{C\subseteq\mathbb{P}:C\text{ is a $\leq$-cover}.\}\]
\end{dfn}

Note that the $\Rightarrow$ part of the definition actually holds for arbitrary $C$.  So $C$ is a $\leq$-cover iff the $\Leftarrow$ part holds which, more explicitly, means that, for all $p,q\in\mathbb{P}$,
\[\tag{$\leq$-Cover$'$}\label{leqCover'}p\nleq q\qquad\Rightarrow\qquad\exists c\in C\ \exists r\in\mathbb{P}\ (p,c\geq r\nleq q).\]
When $\mathbb{P}$ is a frame, i.e. a complete lattice in which finite meets distribute over arbitrary joins (see \cite{PicadoPultr2012}), $\leq$-covers are the same as $\bigvee$-covers.

\begin{prp}\label{FrameCovers}
If $\mathbb{P}$ is a frame then $\mathcal{C}_\leq(\mathbb{P})=\{C\subseteq\mathbb{P}:\bigvee C=1\}$.
\end{prp}

\begin{proof}
If $\bigvee C=1$ and $p^\geq\cap C^\geq\subseteq q^\geq$ then $p=p\wedge\bigvee C=\bigvee_{c\in C}p\wedge c\leq q$ so $C$ is a $\leq$-cover.  Conversely, if $C$ is a $\leq$-cover, $1^\geq\cap C^\geq\subseteq(\bigvee C)^\geq$ implies $1\leq\bigvee C$.
\end{proof}

In particular, \autoref{FrameCovers} applies to spatial frames, i.e. if $\mathbb{P}=\mathcal{O}(X)$, for some space $X$, then $C\subseteq\mathbb{P}$ is a $\subseteq$-cover iff $\bigcup C=X$.  So $\subseteq$-covers are $X$-covers in the usual sense.  In fact, here $\mathbb{P}$ does not have to consist of all open subsets.

\begin{dfn}
We call $\mathbb{P}\subseteq\mathcal{P}(X)$ a \emph{$T_D$ family} if
\[\tag{$T_D$}\forall x\in X\ \exists O,N\in\mathbb{P}\ (O\setminus N=\{x\}).\]
\end{dfn}

Note that $\mathcal{O}(X)$ is a $T_D$ family iff $X$ is a $T_D$ space in the sense of \cite[I.2.1]{PicadoPultr2012}.

\begin{prp}
If $\mathbb{P}\subseteq\mathcal{O}(X)$ then
\begin{align}
\label{XCover=>leqCover}\mathbb{P}\text{ is a basis}\qquad&\Rightarrow\qquad\mathcal{C}_X(\mathbb{P})\subseteq\mathcal{C}_\subseteq(\mathbb{P}).\\
\label{XCover<=leqCover}\mathbb{P}\text{ is }T_D\qquad&\Rightarrow\qquad\mathcal{C}_X(\mathbb{P})\supseteq\mathcal{C}_\subseteq(\mathbb{P}).
\end{align}
\end{prp}

\begin{proof}\
\begin{itemize}
\item[\eqref{XCover=>leqCover}]  Say $\mathbb{P}\subseteq\mathcal{O}(X)$ is a basis for $X$, $\bigcup C=X$ and we have some $O,N\in\mathbb{P}$ with $O^\supseteq\cap C^\supseteq\subseteq N^\supseteq$.  For every $x\in O$, we have some $M\in C$ with $x\in M$.  As $\mathbb{P}$ is a basis, we have some $P\in\mathbb{P}$ with $O,M\supseteq P\ni x$.  By assumption, $x\in P\subseteq N$.  As $x$ was arbitrary, $O\subseteq N$ and hence $C$ is a $\subseteq$-cover.

\item[\eqref{XCover<=leqCover}]  If we had $\bigcup C\neq X$ then we could take some $x\in X\setminus\bigcup C$.  If $\mathbb{P}$ is also $T_D$ then we have some $O,N\in\mathbb{P}$ with $O\setminus N=\{x\}$.  Then $O^\supseteq\cap C^\supseteq\subseteq N^\supseteq$, as $x\notin\bigcup C$, even though $O\nsubseteq N$, showing that $C$ could not be a $\subseteq$-cover.\qedhere
\end{itemize}
\end{proof}

So for $T_D$ bases, spatial covers and order covers are the same thing.  The only drawback here is that $T_D$ bases tend to have rather large cardinality.

\begin{prp}
If $X$ is a compact $T_1$ space with a $T_D$ basis $\mathbb{P}$ then $|\mathbb{P}|\geq|X|$.
\end{prp}

\begin{proof}
If $\mathbb{P}$ is finite then so is $X$, and the only finite $T_1$ spaces are discrete.  So in this case any basis must contain all the singletons and hence $|\mathbb{P}|\geq|X|$.  If $\mathbb{P}$ is infinite then we can take the closure under finite unions without changing $|\mathbb{P}|$.  As $\mathbb{P}$ is $T_D$, for any $x\in X$, we have $O,N\in\mathbb{P}$ with $O\setminus N=\{x\}$.  As $X\setminus O$ is closed and hence compact and $X$ is $T_1$ with basis $\mathbb{P}$, we can cover $X\setminus O$ with finite $F\subseteq\mathbb{P}$ which do not contain $x$.  Thus $N\cup\bigcup F=X\setminus\{x\}\in\mathbb{P}$, as we took the closure under finite unions.  But this means $|\mathbb{P}|\geq|\{X\setminus\{x\}:x\in X\}|\geq|X|$.
\end{proof}

Here is another simple observation about covers.

\begin{prp}\label{Maximum<=>Cover}
$\{p\}\in\mathcal{C}_\leq(\mathbb{P})$ iff $p$ is a maximum of $\mathbb{P}$.
\end{prp}

\begin{proof}
If $p$ is a maximum of $\mathbb{P}$ then $q\leq r$ iff $q^\geq\cap p^\geq=q^\geq\subseteq r^\geq$ so $\{p\}\in\mathcal{C}_\leq(\mathbb{P})$.  Otherwise, we have $q\nleq p$ even though $q^\geq\cap p^\geq\subseteq p^\geq$, so $\{p\}\notin\mathcal{C}_\leq(\mathbb{P})$.
\end{proof}

In particular, if $\mathbb{P}$ has a minimum $0$ then the only way we could have $\{0\}\in\mathcal{C}_\leq(\mathbb{P})$ is if $0$ is also a maximum and hence the only element of $\mathbb{P}$, i.e.
\begin{equation}\label{0Cover}
\{0\}\in\mathcal{C}_\leq(\mathbb{P})\qquad\Leftrightarrow\qquad\mathbb{P}=\{0\}.
\end{equation}

Whenever $\Theta$ can be faithfully represented as a family of basic open covers, like in \autoref{ThetaLocallyCompact} above, it follows that $\Theta\subseteq\mathcal{C}_\leq(\mathbb{P})$, by \eqref{XCover=>leqCover}.  In fact, this applies more generally to any Wallman admissible filter.

\begin{prp}\label{WallmanCovers}
If $\Theta$ is a Wallman admissible filter then $\Theta\subseteq\mathcal{C}_\leq(\mathbb{P})$.
\end{prp}

\begin{proof}
If $p\nleq q$ then, as $\Theta$ is Wallman admissible, we have $C\in\Theta$ such that $(C\setminus\{p\})\cup\{q\}\notin\Theta$.  For any $D\in\Theta$, we have $E\in\Theta$ with $E\leq C$ and $E\leq D$, as $\Theta$ is a directed.  As $(C\setminus\{p\})\cup\{q\}\notin\Theta=\Theta^\leq$, we must have $e\in E$ with $e\nleq(C\setminus\{p\})\cup\{q\}$.  As $e\leq C$, this implies $e\leq p$.  As $e\leq D$, we thus have $d\in D$ with $p,d\geq e\nleq q$.  Thus $D\in\mathcal{C}_\leq(\mathbb{P})$ and hence $\Theta\subseteq\mathcal{C}_\leq(\mathbb{P})$.
\end{proof}

\section{Picado-Pultr Admissibility}\label{Picado-Pultr Admissibility}

We define Picado-Pultr admissibility of $\Theta$ just like \eqref{WallmanAdmissible}, except that we replace the last $\Theta$ with $\mathcal{C}_\leq(\mathbb{P})$, i.e.
\[\label{PicadoPultrAdmissible}\tag{Picado-Pultr Admissible}p\leq q\quad\Leftrightarrow\quad\forall C\in\Theta\ ((C\setminus\{p\})\cup\{q\}\in\mathcal{C}_\leq(\mathbb{P})).\]

\begin{prp}\label{PicadoPultrCovers}
If $\Theta$ is Picado-Pultr admissible then $\Theta\subseteq\mathcal{C}_\leq(\mathbb{P})$.
\end{prp}

\begin{proof}
If $\mathbb{P}=\emptyset$ or $\mathbb{P}=\{0\}$ then $\mathcal{C}_\leq(\mathbb{P})=\mathcal{P}(\mathbb{P})\supseteq\Theta$.  If $\mathbb{P}\geq2$ then $\emptyset\notin\Theta$ \textendash\, otherwise \eqref{PicadoPultrAdmissible} with $p=q$ and $C=\emptyset$ would yield $\{p\}\in\mathcal{C}_\leq(\mathbb{P})$ and hence, by \autoref{Maximum<=>Cover}, $p$ would a maximum of $\mathbb{P}$, for every $p\in\mathbb{P}$, a contradiction.  Thus, for any $C\in\Theta$, we can apply \eqref{PicadoPultrAdmissible} with $p=q\in C$ to show that $C\in\mathcal{C}_\leq(\mathbb{P})$, i.e. $\Theta\subseteq\mathcal{C}_\leq(\mathbb{P})$.
\end{proof}

\begin{cor}
If $\Theta$ is an up-set then
\[\eqref{PicadoPultrAdmissible}\qquad\Rightarrow\qquad\eqref{WallmanAdmissible}.\]
\end{cor}

\begin{proof}
If $\Theta$ is an up-set, the $\Rightarrow$ part of \eqref{WallmanAdmissible} is immediate.  If $\Theta$ is also Picado-Pultr admissible then the $\Leftarrow$ part of \eqref{WallmanAdmissible} follows from the $\Leftarrow$ part of \eqref{PicadoPultrAdmissible}, by \autoref{PicadoPultrCovers}.
\end{proof}

The original admissibility condition defining weak nearness frames in \cite[VIII.4.2]{PicadoPultr2012} involves general sublocales.  In \cite[Corollary VIII.4.2.2]{PicadoPultr2012}, an equivalent condition is given in terms of open and closed sublocales, which we now show is equivalent to \eqref{PicadoPultrAdmissible}.  Following the notation there for frame $\mathbb{P}$, we denote the closed and open sublocales defined by $p\in\mathbb{P}$ by
\begin{align*}
\mathfrak{c}(p)&=p^\leq\qquad\text{and}\\
\mathfrak{o}(p)&=\{p\rightarrow q:q\in\mathbb{P}\}=\{q\in\mathbb{P}:p\rightarrow q=q\}.
\end{align*}
These are considered as elements within the co-frame $\mathcal{S\!\ell}(\mathbb{P})$ of all sublocales.

\begin{thm}\label{PPEquiv}
When $\mathbb{P}$ is a frame and $\Theta$ is an up-set, $\Theta$ is Picado-Pultr admissible iff $\Theta\subseteq\mathcal{C}_\leq(\mathbb{P})$ and, for any $p\in\mathbb{P}$ and $Q=\{q\in\mathbb{P}:\exists C\in\Theta\ (C\setminus q^\geq\leq p)\}$,
\begin{equation}\label{SublocaleAdmissible}
\mathfrak{o}(p)=\bigvee_{q\in Q}\mathfrak{c}(q).
\end{equation}
\end{thm}

\begin{proof}
Note that $q\leq q'$ implies $C\setminus q'^\geq\subseteq C\setminus q^\geq$, and hence $q\in Q$ implies $\mathfrak{c}(q)=q^\leq\subseteq Q$.  By \cite[III.3.3.2]{PicadoPultr2012}, a join of sublocales consists of all the meets of elements in the sublocales.  Thus \eqref{SublocaleAdmissible} is saying that
\begin{equation}\label{SublocaleAdmissibleEquiv}
r=p\rightarrow r\qquad\Leftrightarrow\qquad\exists R\subseteq Q\ (r=\bigwedge R).
\end{equation}

Say $\Theta$ satisfies \eqref{SublocaleAdmissible} and take $q\ngeq p$.  Then $1\neq p\rightarrow q\in\mathfrak{o}(p)$.  By \eqref{SublocaleAdmissibleEquiv}, we have $R\subseteq Q$ with $1\neq p\rightarrow q=\bigwedge R$.  In particular, we have $q'\in R\subseteq Q$ such that $p\rightarrow q\leq q'\neq 1$.  Taking $C\in\Theta$ with $C\setminus q'^\geq\leq p$, we have $C\leq\{p,q'\}\in\Theta$, as $\Theta$ is an up-set.  But $(\{p,q'\}\setminus\{p\})\cup\{q\}=\{q,q'\}\leq q'\neq 1$ so $(\{p,q'\}\setminus\{p\})\cup\{q\}\notin\mathcal{C}_\leq(\mathbb{P})$.  This verifies the $\Leftarrow$ part of \eqref{PicadoPultrAdmissible}.  But if $\Theta\subseteq\mathcal{C}_\leq(\mathbb{P})$ and $\Theta$ is an up-set then the $\Rightarrow$ part of \eqref{PicadoPultrAdmissible} is immediate.

Conversely, say $\Theta$ is Picado-Pultr admissible so $\Theta\subseteq\mathcal{C}_\leq(\mathbb{P})$, by \autoref{PicadoPultrCovers}.  For the $\Leftarrow$ part of \eqref{SublocaleAdmissibleEquiv}, take $q\in Q$, so we have $C\in\Theta\subseteq\mathcal{C}_\leq(\mathbb{P})$ with $C\setminus q^\geq\leq p$ and hence
\[1=\bigvee C\leq q\vee\bigvee(C\setminus q^\geq)\leq q\vee p.\]
Thus $q=(q\vee p)\wedge(p\rightarrow q)=p\rightarrow q$, by \cite[III.3.1.1(H8)]{PicadoPultr2012}.  Consequently, for any $R\subseteq Q$, $p\rightarrow(\bigwedge R)=\bigwedge_{r\in R}(p\rightarrow r)=\bigwedge R$, by \cite[III.3.1]{PicadoPultr2012}, as required.

For the $\Rightarrow$ part of \eqref{SublocaleAdmissibleEquiv}, take any $r=p\rightarrow r$ and let
\[r'=\bigwedge(Q\cap r^\geq).\]
Certainly $r\leq r'$, what we need to show is $r'\leq r$.  If we had $r'\nleq r$ then we would also have $r'\wedge p\nleq r$ (otherwise $r'\wedge p\leq r$ would imply $r'\leq p\rightarrow r=r$).  By \eqref{PicadoPultrAdmissible}, we would then have $C\in\Theta\subseteq\mathcal{C}_\leq(\mathbb{P})$ with
\begin{equation}\label{qneq1}
r\vee\bigvee(C\setminus\{r'\wedge p\})\neq1.
\end{equation}
Taking $q=r\vee\bigvee(C\setminus\{r'\wedge p\})\geq r$, we have $C\setminus q^\geq\subseteq\{r'\wedge p\}\leq p$ so $q\in Q$.  Thus $r'\leq q$ so $q=r'\vee q\geq r'\vee\bigvee(C\setminus\{r'\wedge p\})\geq\bigvee C=1$, contradicting \eqref{qneq1}.
\end{proof}

We finish with a discussion of `compatibility', which gives motivation for Picado-Pultr admissibility.  First, for $C\subseteq\mathcal{P}(X)$, recall that $C_x=\{c\in C:x\in c\}$ so
\[\mathcal{O}(X)_x^\subseteq=\{N:x\in O\in\mathcal{O}(X)\text{ and }O\subseteq N\subseteq X\}.\]
Also recall that a \emph{neighbourhood base} at $x\in X$ is a $\subseteq$-coinitial subset $\mathcal{O}(X)_x^\subseteq$.

\begin{dfn}\label{Compatibility}
We say $\Theta\subseteq\mathcal{P}(\mathcal{P}(X))$ is \emph{compatible} with the topology of $X$ if every $x\in X$ has $(\bigcup C_x)_{C\in\Theta}$ as a neighbourhood base.
\end{dfn}

\begin{rmk}
Again we are following the terminology of \cite[Definition 1.1 (b)]{CollinsHendrie2000} and \cite[Definition 1.13]{Morita1989} (while Morita instead says that $\Theta$ `agrees with the topology' of $X$ in \cite{Morita1951}).  We also note in passing that when $\Theta\subseteq\mathcal{P}(\mathcal{O}(X))$ is countable and compatible it is sometimes called a `development'.
\end{rmk}

Compatibility is closely related to roundness, as the following result shows.

\begin{prp}\label{XCompatible}
If $\mathbb{P}\subseteq\mathcal{O}(X)$ then $\Theta\subseteq\mathcal{P}(\mathbb{P})$ is compatible with the topology of $X$ if and only if $\mathbb{P}$ is a basis, $\Theta\subseteq\mathcal{C}_X(\mathbb{P})$ and $\mathbb{P}_x$ is $\Theta$-round, for each $x\in X$.
\end{prp}

\begin{proof}
If $\Theta$ is compatible with $X$ then, for any $C\in\Theta$ and $x\in X$, $\bigcup C_x\in\mathcal{O}(X)_x$ so $C_x\neq\emptyset$, i.e. $C$ covers $X$.  Also, whenever $x\in O\in\mathcal{O}(X)$, $\subseteq$-coinitiality in $\mathcal{O}(X)_x$ means we have some $C\in\Theta$ with $x\in\bigcup C_x\subseteq O$.  In particular, we have some $N\in C\subseteq\mathbb{P}$ with $x\in N\subseteq O$.  As $O$ was arbitrary, $\mathbb{P}$ (and even $\bigcup\Theta$) is a basis for $X$.  Moreover, taking $O\in\mathbb{P}_x$ shows that $\mathbb{P}_x$ is $\Theta$-round.

Conversely, say $\mathbb{P}$ is a basis of $X$, each $C\in\Theta$ covers $X$ and $\mathbb{P}_x$ is $\Theta$-round, for each $x\in X$.  As each $C\in\Theta$ covers $X$, for each $x\in X$, $C_x\neq\emptyset$ and hence $\bigcup C_x\in\mathcal{O}(X)_x$.  As $\mathbb{P}$ is a basis, for each $x\in X$ and $O\in\mathcal{O}(X)$, we have some $N\in\mathbb{P}$ with $x\in N\subseteq O$.  As $\mathbb{P}_x$ is $\Theta$-round, we have some $C\in\Theta$ such that $x\in\bigcup C_x\subseteq N\subseteq O$ and hence $\Theta$ is compatible with $X$.
\end{proof}

Thus, for any basis $\mathbb{P}\subseteq\mathcal{O}(X)$ and $\Theta\subseteq\mathcal{C}_X(\mathbb{P})$, compatibility is precisely what is necessary to ensure that $\mathbb{P}_x\in\widehat{\Theta}$, for each $x\in X$.  This means that $\widehat\Theta$ can be considered as an `extension' or `completion' of $X$.

It also follows from \autoref{LocalT1Recovery} and \autoref{XCompatible} that any $*$-coinitial directed $\Theta\subseteq\mathcal{C}_X(\mathbb{P})$ is compatible.  In fact, directedness is unnecessary.

\begin{prp}
If $\mathbb{P}\subseteq\mathcal{O}(X)$ is a basis of a $T_1$ space $X$ and $\Theta\subseteq\mathcal{C}_X(\mathbb{P})$ is $*$-coinitial then $\Theta$ is compatible with the topology of $X$.
\end{prp}

\begin{proof}
By \autoref{XCompatible}, it suffices to show that, for any $x\in X$, $\mathbb{P}_x$ is $\Theta$-round.  As $X$ is $T_1$, for any $p\in\mathbb{P}_x$ we have $C\in\mathcal{C}_X(\mathbb{P})$ with $C\cap\mathbb{P}_x=\{p\}$.  By $*$-coinitiality, we have $D\in\Theta$ with $D\cap\mathbb{P}_x\subseteq D*p\cap\mathbb{P}_x\leq C\cap\mathbb{P}_x=\{p\}$, as required.
\end{proof}

Intuitively, Picado-Pultr admissibility should correspond to topological compatibility.  Under suitable conditions, this is indeed the case.

\begin{prp}
If $\mathbb{P}\subseteq\mathcal{O}(X)$ is $T_D$ and $\Theta$ is an up-set then $\Theta$ is compatible with $X$ if and only if $\mathbb{P}$ is a basis and $\Theta\subseteq\mathcal{C}_X(\mathbb{P})$ is Picado-Pultr admissible.
\end{prp}

\begin{proof}
If $\Theta$ is compatible with $X$ then $\mathbb{P}$ is a basis, $\Theta\subseteq\mathcal{C}_X(\mathbb{P})$ and each $\mathbb{P}_x$ is $\Theta$-round, by \autoref{XCompatible}.  Now take $p,q\in\mathbb{P}$ with $p\nsubseteq q$, so we have $x\in p\setminus q$.  As $\mathbb{P}_x$ is $\Theta$-round and $\Theta$ is an up-set, we have $C\in\Theta$ with $C_x=\{p\}$.  So $x\notin q\cup\bigcup(C\setminus\{p\})$ so $C\notin\mathcal{C}_X(\mathbb{P})\supseteq\mathcal{C}_\subseteq(\mathbb{P})$, by \eqref{XCover<=leqCover}, i.e. $\Theta$ is Picado-Pultr admissible.

Conversely, say $\mathbb{P}$ is a basis, $\Theta\subseteq\mathcal{C}_X(\mathbb{P})$ and $\Theta$ is Picado-Pultr admissible.  Whenever $x\in O\in\mathcal{O}(X)$, we have $p,q\in\mathbb{P}$ with $p\setminus q=\{x\}$, as $\mathbb{P}$ is $T_D$.  As $\mathbb{P}$ is basis, we can replace $p$ if necessary to make $p\subseteq O$.  As $\Theta$ is Picado-Pultr admissible, we have $C\in\Theta$ such that $(C\setminus\{p\})\cup\{q\}\notin\mathcal{C}_\subseteq(\mathbb{P})\supseteq\mathcal{C}_X(\mathbb{P})$, by \eqref{XCover=>leqCover}.  As $p\setminus q=\{x\}$, this implies $C_x=\{p\}$, showing that $\Theta$ is compatible with $X$.
\end{proof}

For Wallman admissibility, however, we do not need $T_D$.

\begin{prp}
If $\mathbb{P}\subseteq\mathcal{O}(X)$ then any up-set $\Theta\subseteq\mathcal{C}_X(\mathbb{P})$ that is compatible with $X$ is necessarily Wallman admissible.
\end{prp}

\begin{proof}
If $\Theta\subseteq\mathcal{C}_X(\mathbb{P})$ is compatible with $X$ then each $\mathbb{P}_x$ is $\Theta$-round, by \autoref{XCompatible}.  Now take $p,q\in\mathbb{P}$ with $p\nsubseteq q$, so we have $x\in p\setminus q$.  As $\mathbb{P}_x$ is $\Theta$-round and $\Theta$ is an up-set, we have $C\in\Theta$ with $C_x=\{p\}$.  Thus $x\notin q\cup\bigcup(C\setminus\{p\})$ so $(C\setminus\{p\})\cup\{q\}\notin\mathcal{C}_X(\mathbb{P})\supseteq\Theta$, i.e. $\Theta$ is Wallman admissible.
\end{proof}

\section{Uniformly Below}\label{UniformlyBelow}

The `uniformly below' relation from \cite[VIII.2.3]{PicadoPultr2012} plays an important role when dealing with metric spaces and more general $T_3$ spaces(=$T_1$ regular spaces, i.e. where each point has a neighbourhood base of closed sets).

\begin{dfn}
The \emph{uniformly below} relation $\vartriangleleft$ (or $\vartriangleleft_\Theta$) on $\mathbb{P}$ is given by
\[\tag{Uniformly Below}p\vartriangleleft q\qquad\Leftrightarrow\qquad\exists C\in\Theta\ (Cp\leq q).\]
\end{dfn}

From \eqref{BoundedNear}, we immediately see that $\vartriangleleft$ is auxiliary to $\leq$, i.e.
\[\label{Auxiliary}\tag{Auxiliary}p\leq p'\vartriangleleft q'\leq q\qquad\Rightarrow\qquad p\vartriangleleft q.\]
In particular, if $\vartriangleleft\ \subseteq\ \leq$ then we immediately see that $\vartriangleleft$ is transitive.  Transitivity also follows immediately when $\Theta$ is directed.

\begin{prp}\label{Directed=>Transitive}
If $\Theta$ is directed then $\vartriangleleft$ is transitive.
\end{prp}

\begin{proof}
If $p\vartriangleleft q\vartriangleleft r$ then we have $C,D\in\Theta$ with $Cp\leq q$ and $Dq\leq r$.  As $\Theta$ is directed, we have $E\in\Theta$ with $E\leq C$ and $E\leq D$.  Thus $Ep\leq Cp\leq q$ and $Eq\leq Dq\leq r$ and hence $Ep\subseteq Eq\leq r$, i.e. $p\vartriangleleft r$.
\end{proof}

As in \cite{Erne1991}, we define the \emph{lower preorder} $\trianglelefteq$ on $\mathbb{P}$ by
\[p\trianglelefteq q\qquad\Leftrightarrow\qquad p^\vartriangleright\subseteq q^\vartriangleright.\]
From \eqref{Auxiliary}, we immediately see that $\trianglelefteq$ is weaker than $\leq$, i.e.
\[\label{leqsubtri}\tag{$\leq\ \subseteq\ \trianglelefteq$}p\leq q\qquad\Rightarrow\qquad p\trianglelefteq q.\]
We will be particularly interested in $\Theta$ for which the converse also holds, i.e. $\leq\ =\ \trianglelefteq$.

\begin{prp}\label{NearnessProperties1}
If $\Theta$ is directed and $\trianglelefteq\ =\ \leq$ then $\Theta\subseteq\mathcal{C}_\leq(\mathbb{P})$.
\end{prp}

\begin{proof}
Say $\Theta$ is directed and $\leq\ =\ \trianglelefteq$ and take $C\in\Theta$.  If $p\nleq q$ then we have $s\vartriangleleft p$ with $s\not\vartriangleleft q$.  Thus we have $D\in\Theta$ with $Ds\leq p$.  As $\Theta$ is directed, we have $E\in\Theta$ with $E\leq C$ and $E\leq D$.  Thus $Es\leq p$ but $Es\not\leq q$, as $s\not\vartriangleleft q$, i.e. we have $e\in Es\leq p$ with $e\nleq q$.  We also have $c\in C$ with $e\leq c$, as $E\leq C$.  Thus $r=e$ witnesses \eqref{leqCover'} so $C\in\mathcal{C}_\leq(\mathbb{P})$.
\end{proof}

When $\Theta\subseteq\mathcal{C}_\leq(\mathbb{P})$, weak admissibility means $\vartriangleleft$ is weaker than $\leq$.

\begin{prp}\label{TriangleSubsetLeq}
If $\Theta\subseteq\mathcal{C}_\leq(\mathbb{P})$ then
\[\Theta^\leq\text{ is weakly admissible}\qquad\Leftrightarrow\qquad\vartriangleleft\ \subseteq\ \leq.\]
\end{prp}

\begin{proof}
If $\Theta^\leq$ is not weakly admissible then we have $p,q\in\mathbb{P}$ such that $p\nleq q$ and $\{p\}$ is not $\Theta^\leq$-near.  This means $\mathbb{P}p=\emptyset\leq q$ so $p\vartriangleleft q$ and hence $\vartriangleleft\ \nsubseteq\ \leq$.

Conversely, if $C\in\Theta\subseteq\mathcal{C}_\leq(\mathbb{P})$ and $p\nleq q$ then, by \eqref{leqCover'}, we have $c\in C$ and $r\in\mathbb{P}$ with $p,c\geq r\nleq q$.  If $\Theta^\leq$ is also weakly admissible then $\{r\}$ is $\Theta^\leq$-near so $c\in Cp\nleq q$.  As $C$ was arbitrary, this shows $p\nleq q$ implies $p\not\vartriangleleft q$, i.e. $\vartriangleleft\ \subseteq\ \leq$.
\end{proof}

As in \cite[VIII.2.4]{PicadoPultr2012}, we call $\Theta$ \emph{admissible} if, for all $p\in\mathbb{P}$,
\[\label{triAdmissible}\tag{Admissible}p=\bigvee p^\vartriangleright.\]
In particular, this yields $p^\vartriangleleft\leq p$, i.e. $\vartriangleleft\ \subseteq\ \leq$.  If $p\trianglelefteq q$, i.e. $p^\vartriangleright\subseteq q^\vartriangleright$, it also yields $p=\bigvee p^\vartriangleright\leq\bigvee q^\vartriangleright=q$, i.e. $\trianglelefteq\ \subseteq\ \leq$.  In other words,
\begin{equation}\label{triAdmissible=>leq=trieq}
\eqref{triAdmissible}\qquad\Rightarrow\qquad\vartriangleleft\ \subseteq\ \leq\ =\ \trianglelefteq.
\end{equation}
If $\Theta$ is also directed then \autoref{NearnessProperties1} yields $\Theta\subseteq\mathcal{C}_\leq(\mathbb{P})$.

\begin{rmk}
Together with \autoref{FrameCovers}, this shows that the assumption in \cite[VIII.2.6]{PicadoPultr2012} that a nearness must consist of covers is actually superfluous.  Indeed, while $Cp$ is defined differently in \cite{PicadoPultr2012} as $\{c\in C:c\wedge p\neq0\}$, the proof of \autoref{NearnessProperties1} still applies to show that, if $\Theta$ is directed and admissible,
\[\Theta\subseteq\mathcal{C}_\leq(\mathbb{P})=\{C\subseteq\mathbb{P}:\bigvee C=1\}.\]
\end{rmk}

For filters, \eqref{triAdmissible} strengthens \eqref{PicadoPultrAdmissible}.

\begin{prp}\label{Admissible=>PPAdmissible}
If $\Theta$ is a admissible filter then $\Theta$ is Picado-Pultr admissible.
\end{prp}

\begin{proof}
As just mentioned above, $\Theta\subseteq\mathcal{C}_\leq(\mathbb{P})$.  As $\Theta$ is also an up-set, we immediately see that the $\Rightarrow$ part of \eqref{PicadoPultrAdmissible} holds.

Conversely, say $p\nleq q$ and hence, by \eqref{triAdmissible}, we have some $r\vartriangleleft p$ with $r\nleq q$.  So we have $C\in\Theta$ with $Cr\leq p$.  By \autoref{TriangleSubsetLeq}, $\Theta$ is weakly admissible, which means that no element of $C\setminus Cr$ shares a non-zero lower bound with $r$.  Letting $D=(C\setminus Cr)\cup\{p\}\in\Theta$ and $E=(D\setminus\{p\})\cup\{q\}=(C\setminus Cr)\cup\{q\}$, this means that $r^\geq\cap E^\geq\subseteq q^\geq$, even though $r\nleq q$, showing that $E\notin\mathcal{C}_\leq(\mathbb{P})$, thus verifying the $\Leftarrow$ part of \eqref{PicadoPultrAdmissible}.
\end{proof}

Under admissibility, covers can always be `shrunk'.

\begin{prp}\label{RegularCovers}
If $\Theta$ is admissible and $C\in\mathcal{C}_\leq(\mathbb{P})$ then $C^\vartriangleright\in\mathcal{C}_\leq(\mathbb{P})$.
\end{prp}

\begin{proof}
If $p\nleq q$ then we have $c\in C$ and $r\in\mathbb{P}$ with $p,c\geq r\nleq q$.  By admissibility, we then have $s\vartriangleleft r$ with $s\nleq q$.  By \eqref{Auxiliary}, $s\vartriangleleft c$, i.e. $s\in C^\vartriangleright$, and $s\leq p$, as $\vartriangleleft\ \subseteq\ \leq$, and hence $p,s\geq s\nleq q$, as required. 
\end{proof}

Next we note $\vartriangleleft$-regular subsets (i.e. $R\subseteq\mathbb{P}$ with $R\subseteq R^\vartriangleleft$) are related to $\Theta$-round subsets, at least when they are linked like in \cite[Definition III.3.56]{Kunen2011}.

\begin{dfn}
We call $S\subseteq\mathbb{P}$ \emph{linked} if $\{p,q\}$ is $\Theta^\leq$-near, for all $p,q\in S$.
\end{dfn}

\begin{prp}\label{RegularLinked=>Round}
Any $\vartriangleleft$-regular linked $R\subseteq\mathbb{P}$ is $\Theta$-round.
\end{prp}

\begin{proof}
For any $r\in R$, $\vartriangleleft$-regularity yields $q\in R$ with $q\vartriangleleft r$, so we have $C\in\Theta$ with $Cq\leq r$.  As $R$ is linked, $R\cap C\subseteq Cq\leq r$, verifying that $R$ is $\Theta$-round.
\end{proof}

\begin{thm}\label{CountableDirectedAdmissible=>Faithful}
If $\Theta$ is countable, directed and admissible then $\Theta$ is faithful.
\end{thm}

\begin{proof}
Say $p\nleq q$.  By \eqref{triAdmissible}, we have some $r\vartriangleleft p$ with $r\nleq q$.  For any $C\in\Theta\subseteq\mathcal{C}_\leq(\mathbb{P})$ (see \eqref{triAdmissible=>leq=trieq} and \autoref{NearnessProperties1}), \eqref{leqCover'} yields $c\in C$ and $p_1\in\mathbb{P}$ with $r,c\geq p_1\nleq q$.  Applying the same argument to each $C\in\Theta$ we obtain $(p_n)_{n\in\mathbb{N}}$ such that $R=\bigcup_{n\in\mathbb{N}}p_n^\leq$ is $\Theta$-Cauchy and $p\vartriangleright p_j\vartriangleright p_k\nleq q$, for all $j\leq k$.  Thus $R\in\widehat\Theta_p\setminus\widehat\Theta_q$, by \autoref{TriangleSubsetLeq} \autoref{RegularLinked=>Round}, as required.
\end{proof}

Note the above proof shows that, even if we restricted to the $\vartriangleleft$-regular elements of $\widehat\Theta$, we would still get a faithful representation.  One might wonder if all elements of $\widehat\Theta$ are necessarily $\vartriangleleft$-regular.  The following example shows that this need not be the case, even when $\Theta\subseteq\mathcal{F}(\mathbb{P})$ is countable, directed and admissible.

\begin{xpl}
Take $Y,Z\subseteq\mathbb{N}$ such that $Y\setminus Z$, $Z\setminus Y$ and $Y\cap Z$ are all infinite.  Let $X=\mathbb{N}\cup\{y,z\}$ be the space with basis
\[\mathbb{P}=\{\{n\}:n\in\mathbb{N}\}\cup\{\{y\}\cup Y':Y'\subseteq_\mathcal{F}Y\}\cup\{\{z\}\cup Z':Z'\subseteq_\mathcal{F}Z\}.\]
Note $X$ is compact so if we let $\Theta=\mathcal{C}_X(\mathbb{P})\cap\mathcal{F}(\mathbb{P})$ then $\widehat\Theta=\{\mathbb{P}_x:x\in X\}$, by \autoref{T1Recovery}.  However, as $Y\cap Z$ is infinite, $\mathbb{P}_y\cup\mathbb{P}_z$ is linked which means neither $\mathbb{P}_y$ nor $\mathbb{P}_z$ are $\vartriangleleft$-regular and
\[p\vartriangleleft q\qquad\Leftrightarrow\qquad\exists n\in\mathbb{N}\ (p=\{n\}\subseteq q).\]
Still, $\mathbb{P}$ and hence $\Theta$ is countable and also directed and even admissible, as $Y\setminus Z$ and $Z\setminus Y$ are infinite. 
\end{xpl}

However, we can avoid this situation by placing some degree of regularity on $\Theta$, as we discuss in the next section.

But before moving on, we consider the following notion from \cite{Williams1972} and how it relates to compatibility, regularity and admissibility.

\begin{dfn}
We call $\Theta\subseteq\mathcal{P}(\mathcal{P}(X))$ \emph{locally uniform} if
\[\forall x\in X\ \forall C\in\Theta\ \exists D\in\Theta\ (\bigcup(D*\bigcup D_x)\subseteq\bigcup C_x).\]
\end{dfn}

\begin{prp}\label{LocallyUniform<=>Base}
If $\mathbb{P}\subseteq\mathcal{O}(X)$ then $\Theta\subseteq\mathcal{P}(\mathbb{P})$ is locally uniform and compatible with $X$ iff $(\bigcup(C*\bigcup C_x))_{C\in\Theta}$ is a neighbourhood base at each $x\in X$.
\end{prp}

\begin{proof}
If $\Theta$ is compatible with $X$ then, whenever $x\in O\in\mathcal{O}(X)$, we have $C\in\Theta$ with $x\in C_x\subseteq O$.  If $\Theta$ is also locally uniform then we have $D\in\Theta$ with $x\in\bigcup(D*\bigcup D_x)\subseteq\bigcup C_x\subseteq O$, showing that these sets do indeed form a neighbourhood base at $X$.  Conversely, if $(\bigcup(D*\bigcup D_x))_{D\in\Theta}$ forms a neighbourhood base at $x$ then so too do the smaller neighbourhoods $(\bigcup D_x)_{D\in\Theta}$, so $\Theta$ is compatible with $X$.  Moreover, for any $C\in\Theta$, $\bigcup C_x$ is a neighbourhood of $x$ so by assumption we have some $D\in\Theta$ with $\bigcup(D*\bigcup D_x)\subseteq\bigcup C_x$, i.e. $\Theta$ is also locally uniform.
\end{proof}

\begin{prp}\label{LocallyUniform<=>xRegular}
If $\mathbb{P}\subseteq\mathcal{O}(X)$ and $\Theta$ is compatible with $X$ and directed then
\[\Theta\text{ is locally uniform}\qquad\Leftrightarrow\qquad\forall x\in X\ (\mathbb{P}_x\text{ is $\vartriangleleft$-regular}).\]
\end{prp}

\begin{proof}
If $X=\emptyset$ then the result is vacuously true.  Otherwise $C*p=Cp$, for all $C\subseteq\mathbb{P}$ and $p\in\mathbb{P}$, as $\Theta$ is compatible with $X$ and directed \textendash\, see \autoref{Near=>Bounded} and \autoref{XCompatible}.  Also, $\mathbb{P}$ is a basis so, whenever $x\in O\in\mathcal{O}(X)$, we have $p\in\mathbb{P}$ with $x\in p\subseteq O$.  If $\mathbb{P}_x$ is $\vartriangleleft$-regular then we have $q,r\in\mathbb{P}_x$ and $C,D\in\Theta$ with $Cq\leq p$ and $Dr\leq q$.  As $\Theta$ is directed, we have $E\in\Theta$ with $E\leq C$ and $E\leq D$ so $E_x\subseteq Er\leq Dr\leq q$ and hence $E*\bigcup E_x\subseteq Eq\leq Cq\leq p$.  This shows that $\Theta$ is locally uniform, by \autoref{LocallyUniform<=>Base}.  Conversely, if $\Theta$ is locally uniform and $p\in\mathbb{P}_x$ then, by \autoref{LocallyUniform<=>Base}, we have $C\in\Theta$ such that $\bigcup(C*\bigcup C_x)\subseteq p$.  So, for any $q\in C_x$, $Cq\subseteq\bigcup_{c\in C_x}Cc=C*\bigcup C_x\leq p$, i.e. $q\vartriangleleft p$ so $\mathbb{P}_x$ is $\vartriangleleft$-regular.
\end{proof}

\begin{prp}\label{CompatiblePxRegular=>Admissible}
If $\mathbb{P}\subseteq\mathcal{O}(X)$, $\Theta\subseteq\mathcal{P}(\mathbb{P})$ is compatible with the topology of $X$ and $\mathbb{P}_x$ is $\vartriangleleft$-regular, for each $x\in X$, then $\Theta$ is admissible.
\end{prp}

\begin{proof}
If $p\nsubseteq q$ then we have $x\in X$ with $\mathbb{P}_x\in p\setminus q$.  By \autoref{xNear} and \autoref{XCompatible}, for any $C\in\Theta\subseteq\mathcal{C}_X(\mathbb{P})$, we have $c\in C_x\subseteq Cp$ and hence $c\nleq q$, i.e. $p\not\vartriangleleft q$.  Thus $\bigcup p^\vartriangleright\subseteq p$, for all $p\in\mathbb{P}$, while the reverse inclusion follows from the fact that each $\mathbb{P}_x$ is $\vartriangleleft$-regular, i.e. $p=\bigcup p^\vartriangleright=\bigvee p^\vartriangleright$ so $\Theta$ is admissible.
\end{proof}

\section{Star-Regularity}

\begin{dfn}
We call $\Theta$ \emph{star-regular} if
\[\label{LocallyRegular}\tag{Star-Regular}\forall p\in\mathbb{P}\ \forall C\in\Theta\ \exists D\in\Theta\ (Dp\vartriangleleft C).\]
\end{dfn}

Note that if $\mathbb{P}$ has a maximum $1$, e.g. if $\mathbb{P}$ is a frame, and $\Theta$ is weakly admissible then $D1=D$ or $D\setminus\{0\}$ so star-regularity reduces to $\vartriangleleft$-regularity, i.e. $\Theta\subseteq\Theta^\vartriangleleft$.  Indeed, we could just consider $\vartriangleleft$-regularity, but we felt it worth pointing out that star-regularity suffices for the main results.  Moreover, star-regularity arises naturally from $*$-coinitiality, at least for $T_3$ spaces.
\begin{prp}
If $\mathbb{P}$ is a basis of a $T_3$ space $X$ then
\[\Theta\subseteq\mathcal{C}_X(\mathbb{P})\text{ is $*$-coinitial}\qquad\Rightarrow\qquad\Theta\text{ is star-regular}.\]
\end{prp}

\begin{proof}
By the proof of \autoref{LocalT1Recovery}, $(\Theta*p)^\leq=(\Theta p)^\leq=\mathcal{C}_{\mathrm{cl}(p)}(X)$ and hence
\[p\vartriangleleft q\qquad\Leftrightarrow\qquad\mathrm{cl}(p)\subseteq q.\]
Thus, as $\Theta$ is regular, for any $C\in\mathcal{C}_X(\mathbb{P})$, we have $D\in\mathcal{C}_X(\mathbb{P})$ with $D\vartriangleleft C$.  By $*$-coinitiality, for any $p\in\mathbb{P}$, we have $E\in\Theta$ with $Ep\leq Dp\vartriangleleft C$, as required.
\end{proof}

\begin{prp}\label{Regular=>Admissible}
If $\Theta$ is star-regular and directed then
\[\label{triAdmissible'}\eqref{triAdmissible}\qquad\Leftrightarrow\qquad\leq\ =\ \trianglelefteq.\]
\end{prp}

\begin{proof}
The $\Rightarrow$ part was already observed in \eqref{triAdmissible=>leq=trieq}.  Conversely, if $\leq\ =\ \trianglelefteq$ then, as $\vartriangleleft$ is transitive, by \autoref{Directed=>Transitive}, we have $\vartriangleleft\ \subseteq\ \trianglelefteq\ =\ \leq$, i.e. $p^\vartriangleright\leq p$, for all $p\in\mathbb{P}$.  So, by the definition of suprema and the $\leq\ =\ \trianglelefteq$ assumption, we just need to show
\[p^\vartriangleright\leq q\qquad\Rightarrow\qquad p\trianglelefteq q.\]
So say $p^\vartriangleright\leq q$ and take $r\vartriangleleft p$, so we have $C\in\Theta$ such that $Cr\leq p$.  As $\Theta$ is star-regular, we have $D\in\Theta$ such that $Dr\vartriangleleft C$ and hence $Dr\vartriangleleft Cr\leq p$.  By \eqref{Auxiliary}, $Dr\vartriangleleft p$ and hence $Dr\subseteq p^\vartriangleright\leq q$ so $r\vartriangleleft q$.  As $r$ was arbitrary, this shows that $p\trianglelefteq q$, as required.
\end{proof}

Recall that $\dt{\mathbb{P}}=\{p\in\mathbb{P}:\{p\}\text{ is $\Theta^\leq$-near}\}$.

\begin{thm}\label{CauchyComplete}
If $\Theta$ is star-regular, directed and $\vartriangleleft\ \subseteq\ \leq$ then
\begin{equation}\label{StarRegularSpectrum}
\widehat\Theta=\{R\subseteq\dt{\mathbb{P}}:R\text{ is a $\Theta$-Cauchy $\vartriangleleft$-filter}\}.
\end{equation}
Consequently, if $\Theta$ is also countable then $\widehat\Theta$ is completely metrisable.
\end{thm}

\begin{proof}
We first claim that, for any $R\subseteq\mathbb{P}$,
\begin{equation}\label{CC}
R\text{ is linked and $\Theta$-Cauchy}\qquad\Rightarrow\qquad R^\vartriangleleft\in\widehat\Theta.
\end{equation}
To see this, first take any $s\in R^\vartriangleleft$, so we have $r\in R$ with $r\vartriangleleft s$, which means we have $D\in\Theta$ with $Dr\leq s$.  As $R$ and hence $R^\leq$ is linked and $\vartriangleleft\ \subseteq\ \leq$,
\[D\cap R^\vartriangleleft\subseteq D\cap R^\leq\subseteq Dr\leq s,\]
showing that $R^\vartriangleleft$ is $\Theta$-round.

As $\Theta$ is directed, $\Theta\neq\emptyset$ (actually, \eqref{CC} holds even when $\Theta=\emptyset$, as this implies $\emptyset$ is both the only linked subset and the only point of the spectrum).  As $R$ is $\Theta$-Cauchy, $R\neq\emptyset$ so we have $r\in R$.  For any $C\in\Theta$, the star-regularity of $\Theta$ yields $D\in\Theta$ with $Dr\vartriangleleft C$.  As $R$ is $\Theta$-Cauchy and linked, $\emptyset\neq D\cap R\subseteq Dr\vartriangleleft C$ and hence $R^\vartriangleleft\cap C\neq\emptyset$, i.e. $R^\vartriangleleft$ is $\Theta$-Cauchy.  By \eqref{Auxiliary}, $R^\vartriangleleft$ is also $\leq$-closed so $R^\vartriangleleft\in\widehat\Theta$, thus proving \eqref{CC}.

To prove \eqref{StarRegularSpectrum}, note that if $R\subseteq\mathbb{P}$ is a $\vartriangleleft$-filter then $R$ is $\vartriangleleft$-regular and $\vartriangleleft$-closed, and hence $\leq$-closed, by \eqref{Auxiliary}.  Also, $R$ is $\vartriangleleft$-directed and hence directed, as $\vartriangleleft\ \subseteq\ \leq$, so if $R\subseteq\dt{\mathbb{P}}$ too then $R$ is linked.  Thus $R$ is $\Theta$-round, by \autoref{RegularLinked=>Round}, and if $R$ is also $\Theta$-Cauchy then $R\in\widehat\Theta$.  Conversely, every $R\in\widehat\Theta$ is $\Theta^\leq$-near, by \autoref{xNear}, and hence linked and therefore also contained in $\dt{\mathbb{P}}$.  Thus \eqref{CC} yields $R=R^\leq\supseteq R^\vartriangleleft\in\widehat\Theta$ and hence $R=R^\vartriangleleft$, by \autoref{Round<=>Minimal}, i.e. $R$ is $\vartriangleleft$-regular and hence a $\vartriangleleft$-filter, again by \eqref{Auxiliary} and $\vartriangleleft\ \subseteq\ \leq$, proving \eqref{StarRegularSpectrum}.

So, whenever $r\in R\in\widehat\Theta$, we have $s\in R$ with $s\vartriangleleft r$.  This means we have $C\in\Theta$ with $Cs\leq r$ and hence, by \autoref{xNear},
\[\bigcup\{\widehat\Theta_c:c\in C\text{ and }\widehat\Theta_c\cap\widehat\Theta_s\neq\emptyset\}\subseteq\bigcup\{\widehat\Theta_c:c\in Cs\}\subseteq\widehat\Theta_r.\]
If $\Theta$ is also countable then so is $\Theta'=\{\{\widehat\Theta_c:c\in C\}:C\in\Theta\}$ and hence these covers of $\widehat\Theta$ satisfy the required conditions of the Arhangelskii-Stone metrisation theorem (see \cite{CollinsMorita1998}), i.e. $\widehat\Theta$ is metrisable.  Moreover, every $\Theta'$-Cauchy filter in $\{\widehat\Theta_p:p\in\dt{\mathbb{P}}\}$ must be of the form $\{\widehat\Theta_r:r\in R\}$, for some linked $\Theta$-Cauchy $R$, which converges to $R^\vartriangleleft\in\widehat\Theta$, by \eqref{CC}.  So $\widehat\Theta$ is \v{C}ech-complete and hence completely metrisable.
\end{proof}

Even if $\Theta$ is not (locally) regular, it can often be `regularised'.  First define
\[\Theta'=\{C^\vartriangleright:C\in\Theta\}.\]
Setting $\Theta^1=\Theta$ and $\Theta^{n+1}=({\Theta^n})'$ we define the \emph{regularisation} of $\Theta$ by
\[\Theta^\mathsf{R}=\bigcup_{n\in\mathbb{N}}\Theta^n.\]
By definition, $\Theta^\mathsf{R}$ is $\vartriangleleft$-regular.  Also $|\Theta|=|\Theta^\mathsf{R}|$, as long as $\Theta$ is infinite.  

\begin{thm}\label{AdmissibleDirectedRegularisation}
If $\Theta$ is admissible and directed then so is $\Theta^\mathsf{R}$ and
\[\widehat{\Theta^\mathsf{R}}=\{R\subseteq\dt{\mathbb{P}}:R\text{ is a $\Theta$-Cauchy $\vartriangleleft$-filter}\}.\]
If $\Theta$ is countable then again $\widehat{\Theta^\mathsf{R}}$ is completely metrisable.
\end{thm}

\begin{proof}
By \eqref{Auxiliary}, $C\leq D$ implies $C^\vartriangleright\subseteq D^\vartriangleright$.  As $\Theta$ is directed, this implies that $\Theta'$ and hence $\Theta^\mathsf{R}$ is $\subseteq$-directed and, in particular, directed.

By \autoref{Admissible=>PPAdmissible}, $\Theta^\leq$ is Picado-Pultr admissible.  By \autoref{RegularCovers}, $\Theta^\leq\subseteq\Theta'^\leq\subseteq\mathcal{C}_\leq(\mathbb{P})$ so $\Theta'^\leq$ is also Picado-Pultr admissible.  In particular, $\Theta'^\leq$ is weakly admissible so $\vartriangleleft_{\Theta'}\subseteq\ \leq$, by \autoref{TriangleSubsetLeq}.  Also, as $\Theta,\Theta'\subseteq\mathcal{C}_\leq(\mathbb{P})$, the only way $\{0\}$ could be $\Theta^\leq$-near or $\Theta'^\leq$-near is if $\mathbb{P}=\{0\}$, by \autoref{0near} and \eqref{0Cover}, in which case $\Theta=\Theta'$.  So $\dt{\mathbb{P}}$ is the same with respect to $\Theta$ and $\Theta'$, and hence the same goes for finite near subsets and hence stars, by \autoref{FiniteNearBounded}.  Thus, as $\Theta^\leq\subseteq\Theta'^\leq$,
\[p^\vartriangleright\subseteq p^{\vartriangleright_{\Theta'}}\leq p=\bigvee p^\vartriangleright\]
and hence $p=\bigvee p^{\vartriangleright_{\Theta'}}$, for all $p\in\mathbb{P}$, i.e. $\Theta'$ is admissible.  By induction, the same is true of $\Theta^n$, for all $n\in\mathbb{N}$, and hence of $\Theta^\mathsf{R}$.  In particular, $\Theta^\mathsf{R}$ is admissible.

Now say $R\subseteq\dt{\mathbb{P}}$ is a $\Theta$-Cauchy $\vartriangleleft$-filter.  As $R$ is $\vartriangleleft$-regular and $\Theta$-Cauchy, $R$ must be $\Theta'$-Cauchy and by induction it follows that $R$ is $\Theta^\mathsf{R}$-Cauchy.  As $\vartriangleleft\ \subseteq\ \leq$, $R\subseteq\dt{\mathbb{P}}$ is directed and hence linked.  Thus $R$ is $\Theta$-round, by \autoref{RegularLinked=>Round}, and hence $\Theta^\mathsf{R}$-round, as $\Theta\subseteq\Theta^\mathsf{R}$, so $R\in\widehat{\Theta^\mathsf{R}}$.

Conversely, if $R\in\widehat{\Theta^\mathsf{R}}$ then $R$ is $\Theta^\mathsf{R}$-Cauchy and hence $\Theta$-Cauchy, as $\Theta\subseteq\Theta^\mathsf{R}$.  As $R$ is $\Theta^\mathsf{R}$-round, for any $r\in R$, we have $C\in\Theta^\mathsf{R}$ with $C\cap R\leq r$.  As $R$ is $\Theta^\mathsf{R}$-Cauchy, we have some $s\in C^\vartriangleright\cap R\vartriangleleft C\cap R\leq r$ and hence $s\vartriangleleft r$, showing that $R$ is $\vartriangleleft$-regular.  As $\Theta^\mathsf{R}$ is directed, $R$ is also directed, by \autoref{Directed=>Directed}, and hence a $\vartriangleleft$-filter, by \eqref{Auxiliary}.

If $\Theta$ is countable then so is $\Theta^\mathsf{R}$.  Also, $\Theta^\mathsf{R}$ is $\vartriangleleft$-regular and hence $\vartriangleleft_{\Theta^\mathsf{R}}$-regular, as $\Theta\subseteq\Theta^\mathsf{R}$ so $\vartriangleleft\ \subseteq\ \vartriangleleft_{\Theta^\mathsf{R}}$ (as stars are the same relative to $\Theta$ and $\Theta^\mathsf{R}$).  Thus $\widehat{\Theta^\mathsf{R}}$ is completely metrisable, by \autoref{CauchyComplete}.
\end{proof}

Let us call a concrete filter of covers $\Theta\subseteq\mathcal{C}_X(\mathbb{P})$ \emph{complete} if every $\Theta$-Cauchy $\vartriangleleft$-filter $R\subseteq\mathbb{P}$ converges to some $x\in X$, i.e. $\mathbb{P}_x\subseteq R$.  Then we can state our duality for completely metrisable spaces as follows.

\begin{thm}\label{CompletelyMetrisableDualitySummary}
\
\begin{center}
\textbf{\upshape Abstract countable admissible filters $\Theta\subseteq\mathcal{P}(\mathbb{P})$ are dual to concrete countable complete locally uniform compatible filters $\Theta\subseteq\mathcal{C}_X(\mathbb{P})$ where $\mathbb{P}\subseteq\mathcal{O}(X)$ is a basis of a completely metrisable space $X$}.
\end{center}
\end{thm}

More precisely, if $\Theta$ is a countable admissible filter then $\Theta$ can be concretely represented as the basis $(\widehat{\Theta^\mathsf{R}}_p)_{p\in\mathbb{P}}$ of the completely metrisable space $\widehat{\Theta^\mathsf{R}}$ where $\leq$ becomes $\subseteq$ and $\Theta$ becomes the complete locally uniform compatible family of covers $(\{\widehat{\Theta^\mathsf{R}}_c:c\in C\})_{C\in\Theta}$ of $\widehat{\Theta^\mathsf{R}}$, by \autoref{XCompatible}, \autoref{LocallyUniform<=>xRegular}, \autoref{CountableDirectedAdmissible=>Faithful} and \autoref{AdmissibleDirectedRegularisation}.  Conversely, if $\Theta\subseteq\mathcal{C}_X(\mathbb{P})$ is a locally uniform compatible filter then $\Theta$ is admissible, by \autoref{CompatiblePxRegular=>Admissible}.  If $\Theta$ is also complete then by definition (and \autoref{AdmissibleDirectedRegularisation}),
\[\widehat{\Theta^\mathsf{R}}=\{R\subseteq\dt{\mathbb{P}}:R\text{ is a $\Theta$-Cauchy $\vartriangleleft$-filter}\}=\{\mathbb{P}_x:x\in X\},\]
showing that $\widehat{\Theta^\mathsf{R}}$ is homeomorphic to the original space $X$.
\vfill

\bibliography{maths}{}

\begin{thebibliography}{BHH98}

\bibitem[Ale39]{Alexander1939}
James~Waddell Alexander.
\newblock Ordered sets, complexes and the problem of compactification.
\newblock {\em Proceedings of the National Academy of Sciences},
  25(6):296--298, 1939.
\newblock \href {http://dx.doi.org/10.1073/pnas.25.6.296}
  {\path{doi:10.1073/pnas.25.6.296}}.

\bibitem[BHH98]{BentleyHerrlichHusek1998}
H.~L. Bentley, H.~Herrlich, and M.~Hu\v{s}ek.
\newblock The historical development of uniform, proximal, and nearness
  concepts in topology.
\newblock In {\em Handbook of the history of general topology, {V}ol. 2 ({S}an
  {A}ntonio, {TX}, 1993)}, volume~2 of {\em Hist. Topol.}, pages 577--629.
  Kluwer Acad. Publ., Dordrecht, 1998.

\bibitem[BS18a]{BiceStarling2018}
Tristan Bice and Charles Starling.
\newblock General non-commutative locally compact locally {H}ausdorff {S}tone
  duality, 2018.
\newblock \href {http://arxiv.org/abs/1803.00394} {\path{arXiv:1803.00394}}.

\bibitem[BS18b]{BiceStarling2018b}
Tristan Bice and Charles Starling.
\newblock Hausdorff tight groupoids generalized, 2018.
\newblock \href {http://arxiv.org/abs/1809.08578} {\path{arXiv:1809.08578}}.

\bibitem[CH00]{CollinsHendrie2000}
P.~J. Collins and C.~A. Hendrie.
\newblock On strict extensions of nearness spaces.
\newblock {\em Appl. Categ. Structures}, 8(3):435--446, 2000.
\newblock Dedicated to Professor H. Herrlich on the occasion of his sixtieth
  birthday.
\newblock \href {http://dx.doi.org/10.1023/A:1008614929530}
  {\path{doi:10.1023/A:1008614929530}}.

\bibitem[CM98]{CollinsMorita1998}
P.~J. Collins and K.~Morita.
\newblock Coverings versus entourages.
\newblock {\em Topology Appl.}, 82(1-3):125--130, 1998.
\newblock Special volume in memory of Kiiti Morita.
\newblock \href {http://dx.doi.org/10.1016/S0166-8641(97)00054-0}
  {\path{doi:10.1016/S0166-8641(97)00054-0}}.

\bibitem[DT18]{DebskiTymchatyn2018}
Wojciech D\polhk{e}bski and E.~D. Tymchatyn.
\newblock Cell structures and topologically complete spaces.
\newblock {\em Topology Appl.}, 239:293--307, 2018.
\newblock \href {http://dx.doi.org/10.1016/j.topol.2018.02.020}
  {\path{doi:10.1016/j.topol.2018.02.020}}.

\bibitem[Ern91]{Erne1991}
M.~Ern\'{e}.
\newblock The {ABC} of order and topology.
\newblock {\em Category Theory at Work}, 18:57--83, 1991.
\newblock URL: \url{http://www.heldermann.de/R&E/RAE18/ctw05.pdf}.

\bibitem[Har71]{Harris1971}
Douglas Harris.
\newblock {\em Structures in topology}.
\newblock American Mathematical Society, Providence, R.I., 1971.
\newblock Memoirs of the American Mathematical Society, No. 115.

\bibitem[Kun11]{Kunen2011}
Kenneth Kunen.
\newblock {\em Set theory}, volume~34 of {\em Studies in Logic (London)}.
\newblock College Publications, London, 2011.

\bibitem[Mor51]{Morita1951}
Kiiti Morita.
\newblock On the simple extension of a space with respect to a uniformity.
  {I}-{IV}.
\newblock {\em Proc. Japan Acad.}, 27:65--72,130--137,166--171,632--636, 1951.
\newblock URL: \url{http://projecteuclid.org/euclid.pja/1195571517}.

\bibitem[Mor89]{Morita1989}
Kiiti Morita.
\newblock Extensions of mappings. {I}.
\newblock In {\em Topics in general topology}, volume~41 of {\em North-Holland
  Math. Library}, pages 1--39. North-Holland, Amsterdam, 1989.
\newblock \href {http://dx.doi.org/10.1016/S0924-6509(08)70147-8}
  {\path{doi:10.1016/S0924-6509(08)70147-8}}.

\bibitem[PP12]{PicadoPultr2012}
Jorge Picado and Ale{\v{s}} Pultr.
\newblock {\em Frames and locales: Topology without points}.
\newblock Frontiers in Mathematics. Birkh\"auser/Springer Basel AG, Basel,
  2012.
\newblock \href {http://dx.doi.org/10.1007/978-3-0348-0154-6}
  {\path{doi:10.1007/978-3-0348-0154-6}}.

\bibitem[Sco13]{Scott2013}
Brian~M. Scott.
\newblock Compact and locally {H}ausdorff, but not locally compact.
\newblock Mathematics Stack Exchange, 2013.
\newblock URL: \url{https://math.stackexchange.com/q/337657}.

\bibitem[Tuk40]{Tukey1940}
John~W. Tukey.
\newblock {\em Convergence and {U}niformity in {T}opology}.
\newblock Annals of Mathematics Studies, no. 2. Princeton University Press,
  Princeton, N. J., 1940.

\bibitem[Wal38]{Wallman1938}
Henry Wallman.
\newblock Lattices and topological spaces.
\newblock {\em Ann. of Math. (2)}, 39(1):112--126, 1938.
\newblock \href {http://dx.doi.org/10.2307/1968717}
  {\path{doi:10.2307/1968717}}.

\bibitem[Wil72]{Williams1972}
James Williams.
\newblock Locally uniform spaces.
\newblock {\em Trans. Amer. Math. Soc.}, 168:435--469, 1972.
\newblock \href {http://dx.doi.org/10.2307/1996185}
  {\path{doi:10.2307/1996185}}.

\end{thebibliography}
\bibliographystyle{alphaurl}

\end{document}